\documentclass[10pt,a4paper]{article}

\usepackage[utf8]{inputenc}
\usepackage{makeidx}
\usepackage[T1]{fontenc}
\usepackage{amsmath}
\usepackage{amsthm}
\usepackage{amsfonts}
\usepackage{amssymb}
\usepackage{alltt}
\usepackage{color}
\usepackage{graphicx}
\newcommand{\N}{\mathbb N}

\newcommand{\R}{\mathbb R}

\newcommand{\D}{\mathbb D}

\def\P{\mathbb P}

\newcommand{\be}{\begin{equation}}
\newcommand{\ee}{\end{equation}}
\def\1{{\bf 1}}

\def\ep{\epsilon}

\def\Dt0{{\bf D}}

\def\E{{\bf E}}
\def\M{{\bf M}}
\def\P{{\bf P}}

\def\p{{\bf p}}

\def\to{\rightarrow}

\def \to{\rightarrow}

\def \R {\mathbb{R}}
\def \N {\mathbb{N}}

\def \ep{\varepsilon}

\def\E{\mathbb E}
\def\P{\mathbb P}

\definecolor{ProcessBlue}{cmyk}{1,0,0,0.40}

\newtheorem{Theorem}{Theorem}[section]
\newtheorem{Definition}[Theorem]{Definition}
\newtheorem{Proposition}[Theorem]{Proposition}
\newtheorem{Lemma}[Theorem]{Lemma}
\newtheorem{Corollary}[Theorem]{Corollary}
\newtheorem{Remark}{Remark}[section]

\newcommand{\FR}{{\cal F}}

\newcommand{\XR}{{\cal X}}

\title{Mean Field Games in a Stackelberg problem with an informed major player}
\author{Philippe Bergault\thanks{CEREMADE, UMR CNRS 7534, Universit\'e Paris Dauphine-PSL,
Place de Lattre de Tassigny, 75775 Paris Cedex 16, France. e-mail: bergault@ceremade.dauphine.fr} \and 
Pierre Cardaliaguet\thanks{CEREMADE, UMR CNRS 7534, Universit\'e Paris Dauphine-PSL,
Place de Lattre de Tassigny, 75775 Paris Cedex 16, France. e-mail: cardaliaguet@ceremade.dauphine.fr}
 \and Catherine Rainer\thanks{ Univ Brest,  UMR CNRS 6205, Laboratoire de Mathématiques de Bretagne Atlantique, 6, avenue Victor-le-Gorgeu, B.P. 809, 29285 Brest cedex, France. e-mail: Catherine.Rainer@univ-brest.fr }}

\begin{document}

\maketitle

\begin{abstract} We investigate a stochastic differential game in which a major player has a private information (the knowledge of a random variable), which she discloses through her control to a population of small players playing in a Nash Mean Field Game equilibrium. The major player's cost depends on the distribution of the population, while the cost of the population depends on the random variable known by the major player. We show that the game has a relaxed solution and that the optimal control of the major player is approximatively optimal in games with a large but finite number of small players. 
\end{abstract}

\tableofcontents

%\bigskip
%{\color{red} RESTE A FAIRE 
%\begin{itemize}
%%\item ``We claim that the process ${\bf p}^{{\bf u}^0}$ defined in \eqref{defpu0} has law ${\bf P}$ under ${\bf u}^0$. {\color{red} A FAIRE}." dans preuve Prop \ref{prop.equiv}
%%\item Preuve Lemma \ref{lem.compact} (facile)
%\item Preuve Lemma \ref{lem.basic} : passage de la loi sans completion a loi dans filtration complete
%%\item Preuve derniere partie lemme \ref{lem.buildalpha}
%%\item Preuve lemme densite \ref{lem.densite} 
%%\item Dans Lemma \ref{lem.limMFG}, question du couplage alors que $\Dt0$ pas Polonais
%\end{itemize}
%}

\section{Introduction} 

We are interested in a model with a major player and infinitely many small players. In this model, the major player has a private information, which is disclosed to her at the beginning of the game. The game is then played in a Stackelberg equilibrium: the major player announces her strategy and, given this strategy, the small players play a mean field game (MFG) Nash equilibrium. The whole point is that, as the strategy of the major player might depend on her private information, the realization of her control gives a hint on this private information to the small players along the time. 

The paper relies on the literature on mean field games, which investigates differential games with infinitely many players. Since the pioneering works of Lasry and Lions \cite{LL06cr1, LL06cr2, LLJapan, LiCoursCollege} and of Caines, Huang and Malham\'e \cite{huang2006large, HCM07, HCM07-2}, the topic has known a fast growth: we refer the interested reader to the monographs \cite{ACDPS20, CDLL, carmona2013probabilistic} and the references therein. Here we consider an MFG with a major player: such class of models has been introduced by Huang in \cite{H10} and discussed further  in \cite{BCY16, BLL20, CCP20, CW16, CW17, CZ16, MB15, NH12, NC13, SC16}. We follow here the approach of \cite{BCY16, MB15, NH12}, in which the problem is played in a Stackelberg equilibrium, the major player being the leader. A related work is \cite{EMP19}, which makes the link between the litterature of principal-agent problems and MFG. The very interesting and recent paper \cite{Dj23} makes the link between Stackelberg equilibria in games with finitely many players and a major player and with  mean field games with a major player: it shows in particular how to use the optimal strategy of the limit game in the $N-$player game.

Little is known on MFG problems with information issues.  \c{S}en and  Caines  \cite{SC16} and Firoozi and Caines \cite{FiCa20} study a MFG problem in which the small players observe only partially a major agent. The approach uses  nonlinear filtering theory to build an associated completely observed system.  Casgrain and Jaimungal investigate in \cite{CaJa20} a MFG in which the agents have a different belief on what model the real world follows. In \cite{Be22} Bertucci considers MFG problems with a lack of information, in which the players have the same information, and builds an associated master equation. The recent paper \cite{BeReTa} investigates situations where the players have to spend some effort to obtain information on their state. The main difference between the works \cite{BeReTa, Be22, CaJa20, FiCa20, SC16} quoted above and  our framework is that, in our setting, the major player actively manipulates her information in order to reach some goal. 
\bigskip

In order to present our contribution, let us describe a little further our model. We consider a MFG problem with an informed player in a Stackelberg problem. 
\begin{itemize}
\item At time $0$ (before the game starts), nature chooses at random an index ${\bf i}\in \{1, \dots, I\}$ with law $p^0=(p^0_i)_{i=1,\dots, I}$ and announces the result to the major player only. The index ${\bf i}$ is  the private information of the major player. 
\item The goal of the major player is to minimize over her (random) control $({\bf u}^0=({\bf u}^0_i)_{i=1, \dots, I}))$ a cost of the form
$$
\E\left[\int_0^T L^0_{\bf i}(t, {\bf u}^0_{{\bf i},t},{\bf m}_t)dt\ |\ {\bf i}=i\right] 
$$
where ${\bf m}_t$ is the (random) distribution of the small players. 

\item Once the major player has announced her strategy, the small players, observing the realization ${\bf u}^0_{\bf i}$ of the control of the informed player and their own state,  minimize (in a decentralized way through a mean field game) their cost
$$
\E\left[ \int_0^T L_{\bf i}(X_t, \alpha_t)+F_{{\bf i}}(X_t,{\bf m}_t)dt+ G_{{\bf i}}(X_T, {\bf m}_T)\right],
$$
where $(\alpha_t)$ is the control  of a typical small player and $(X_t)$ is the process
$$
X_t=Z+\int_0^t \alpha_sds+ \sqrt{2} B_t, \qquad t\in [0,T]. 
$$
Above, $Z$ is a random variable on $\R^d$ with law $m_0\in  \mathcal P_2(\R^d)$, $B$ is a standard $d-$dimensional Brownian motion, with $Z$, $B$ and $({\bf i},{\bf u}^0)$ independent.   For any $i\in \{1, \dots, I\}$, $L_i : \R^d\times \R^d\to \R$ is the cost for a small player to play a control, $F_i, G_i: \R^d \times  \mathcal P_2(\R^2)\to \R$ are the interaction costs between the small players, $L^0_i : [0,T]\times U^0\times \mathcal P_2(\R^2)\to \R$ is the running cost for the major player. The maps $L_i$, $F_i$, $G_i$ and $L^0_i$ are at least continuous and locally bounded.
\end{itemize}

The fact that the small player's cost is given by a mean field game equilibrium means that the distribution of players ${\bf m}_t$ is the conditional law of the optimal path $X^*$ given ${\bf u}^0$. 

Note that, in contrast with the major player, the minor players do not know the index ${\bf i}$. However they can obtain some information on ${\bf i}$ by observing the major player's actions. 

%Above, $Z$ is a random variable on $\R^d$ with law $m_0\in  \mathcal P_2(\R^d)$, $B$ is a standard $d-$dimensional Brownian motion, with $Z$, $B$ and $({\bf i},{\bf u}^0)$ independent. For any $i\in \{1, \dots, I\}$, $L_i : \R^d\times \R^d\to \R$ is the cost for a small player to play a control, $F_i, G_i: \R^d \times  \mathcal P_2(\R^2)\to \R$ are the interaction costs between the small players, $L^0_i : [0,T]\times U^0\times \mathcal P_2(\R^2)\to \R$ is the running cost for the major player. The maps $L_i$, $F_i$, $G_i$ and $L^0_i$ are at least continuous and locally bounded. \\
%Throughout the paper we assume that the $F_i$ and the $G_i$ are monotone, so that, given the control ${\bf u}^0$ of the major player, the MFG equilibrium is unique. \\

In this very elementary game, interactions between major and minor players are minimal: the small players appears in the major player's cost only through the distribution of their state, while the major player influences the minor player's actions only by revealing (or not) information on ${\bf i}$. This information  structure is inspired by Aumann-Maschler theory of repeated games with incomplete information: see \cite{AuMaSt} for a presentation of the theory and \cite{CaRa09} for its  extension  to two player zero-sum stochastic differential games. \\

Our main results are the following: first we rewrite the problem in a relaxed form and show that the major player has an optimal strategy for the game (Theorem \ref{thm.existence}). This relaxed formulation involves the theory of MFG with a common noise (common noise being here the information disclosed by the major player) as developed in \cite{CSS22}. Second, we show that such an optimal strategy is still approximatively optimal in Stackelberg games with finite many small players (Theorem \ref{thm.main}). The fact that the optimal strategy in a mean field game provides an approximated Nash equilibria for the associated game with finitely many players is very classical: see \cite{carmona2013probabilistic, huang2006large} for instance. The extension of such a property to Stackelberg equilibria in MFG with a large player has only been handled  very recently in \cite[Theorem 2.20]{Dj23} in a very general framework, but quite different from our: indeed in \cite{Dj23}, the major player can choose among the Nash equilibria of the $N-$player game, while in our set-up it is more natural that she has to handle any such Nash equilibrium. This difference leads us to show the approximate optimality of the control for the major player when the small players play {\it any} (approximate) Nash equilibria. To prove such a stability, we use in a very strong way the characterization of the limit of  (approximate and open loop) Nash equilibria  in differential games with a large number of players as discussed by Lacker \cite{La16} and Fischer \cite{Fi17}, as well as the standard Lasry-Lions monotonicity condition which ensures the uniqueness of the MFG Nash equilibrium. The main difference with Lacker (besides  the framework, which is much more general in \cite{La16}) is the fact that, in our case, the common noise is the revelation by the major player of her private knowledge (and thus does not have a fixed law), while in \cite{La16} it is a fixed Brownian motion.  Let us finally underline that, in the game with finitely many players, we allow the small players to play only ``open loop strategies'': following Lacker and Flem \cite{LaFl21} and Djete \cite{Dj21}, an extension to problems in which the small players are allowed to play closed loop strategies is probably doable, but the proof would be much more technical. \\

Let us finally comment upon the choice of Stackelberg versus Nash equilibria. As explained above, the analysis of Nash equilibria for MFG problems with a major player is well understood \cite{CCP20, CW16, CW17, CZ16}. However the analysis of Nash equilibria with incomplete information seems much more challenging. Indeed, in order to study this question, one should understand how the small players could play in such a game without knowing a priori the major player's strategy: this question is understood in a two player zero-sum game framework \cite{CaRa09}, but remains completely open for nonzero-sum differential games.

\section{Notation, assumption and definitions} 
\subsection{Notation and assumptions} 

We  work within the set $\mathcal P(\R^d)$ of Borel probability measures on $\R^d$. The set $\mathcal P_2(\R^d)$ is the subset of $\mathcal P(\R^d)$ of measures with finite second order moment; it is endowed with the Wasserstein distance ${\bf d}_2$. 

We fix $I\in \N$, with $I\geq 2$. We denote by $\Delta(I)$ the simplex $\Delta(I)= \{p\in \R^I_+, \; \sum_i p_i=1\}$ interpreted as the set of probability measures over $\{1, \dots, I\}$. We also fix $p^0\in \Delta(I)$, which is the initial belief of the small players on the random variable ${\bf i}$ (which has law $p^0$). 

Let $\Dt0$ be the set of c\`{a}dl\`{a}g functions from $\R\to \Delta(I)$, endowed with the Meyer-Zheng topology. Let us recall that, if $\Dt0$ is not a Polish space, it is continuously embedded into the Polish space $\mathcal P$ of Borel probability measures on $[0,T]\times \Delta(I)$ with first marginal equal to the Lebesgue measure. Hence any Borel probability measure on $\Dt0$ may be viewed as a Borel probability measure on  the Polish space $\mathcal P$. We use this property throughout the paper. 

% Let $t\mapsto  p(t)$ be the coordinate mapping on $\Dt0$, $\mathcal G$ be the Borel $\sigma-$algebra on $\Dt0$ and  $({\cal G}_t)$ be the filtration generated  by $t\mapsto p(t)$. Given $p_0\in \Delta(I)$, we denote by $\M(p_0)$ the set of probability measures ${\bf P}$ on $\Dt0$ such that, under $\bf P$, $(\p(t),t\in\R)$ is a martingale and $p_t=p_0$ for $t\leq 0$. 
%Finally for any measure ${\bf P}$ on $\Dt0$, we denote by ${\mathcal F}^{\bf P}$ the completion of the filtration $({\cal G}_t)$ with respect to ${\bf P}$ and by 
%$\E^{{\bf P}}[\dots]$ the expectation with respect to ${\bf P}$.

Recall that, for any $i\in \{1, \dots, I\}$, $L_i : \R^d\times \R^d\to \R$ is the cost for a small player to play a control, $F_i, G_i: \R^d \times  \mathcal P_2(\R^2)\to \R$ are the interaction costs between the small players, $L^0_i : [0,T]\times U^0\times \mathcal P_2(\R^2)\to \R$ is the running cost for the major player. Given $p\in \Delta(I)$, $x,u,\xi\in \R^d$ and $m\in \mathcal P(\R^d)$, we set
\be\label{def.LFGp}
L(x,u, p)= \sum_{i=1}^I p_i L_i(x,u), \; F(x,m,p)= \sum_{i=1}^I p_i F_i(x,m), \; G(x,m,p)= \sum_{i=1}^I p_i G_i(x,m), 
\ee
$$
H(x,\xi, p)= \sup_{u\in \R^d} -\xi \cdot u-L(x,u,p), 
$$
$$
L^0(s, u^0, p, m)=  \sum_{i=1}^I p_i L^0_i(s,u^0,p,m), 
\qquad \bar L^0(s, p, m)= \inf_{u^0\in U^0} L^0(s, u^0, p, m).
$$
Let us recall that a continuous, bounded map $K:\R^d\times \mathcal P(\R^d)\to \R$ is monotone if 
$$
\int_{\R^d} (K(x,m_1)-K(x,m_2))(m_1-m_2)(dx) \geq 0\qquad \forall m_1,m_2\in \mathcal P(\R^d).
$$ 
We say that the map $K$ is strongly monotone if there exists $\alpha>0$ such that
\begin{equation}\label{SM1}
\displaystyle \int_{\R^d} (K (x,m_1)-K(x,m_2)) (m_1-m_2)(dx) 
 \geq \alpha \int_{\R^d} (K(x,m_1)-K(x,m_2))^2dx. 
\end{equation}
It is called strictly monotone if 
\begin{equation}\label{SM2}
\displaystyle \int_{\R^d} (K (x,m_1)-K(x,m_2)) (m_1-m_2)(dx)\leq 0 \ \  \text{implies} \ \ m_1=m_2.
\end{equation}
A classical example of a map $K$ satisfying the strict and strong monotonicity conditions, which  goes back to \cite{LL06cr1, LL06cr2, LLJapan, LiCoursCollege}, is  
\be\label{athens1}
K(x,m) =  f (\cdot, m\ast \rho (\cdot)) \ast \rho.
\ee
where $\rho$ is a smooth, non negative and radially symmetric kernel with Fourier transform $\hat \rho$ vanishing almost nowhere, and $ f:  \R^d \times  \R\to \R$ is  strictly increasing and Lipschitz continuous in the second variable, that is,  there exists $\alpha\in (0,1)$ such that 
\[ \alpha \leq \frac{\partial  f}{\partial s}(x,s)\leq \alpha^{-1},\]
with $f(x,0)$ bounded. \\

\noindent {\bf Assumption:} The following conditions are in force throughout the paper:
\begin{itemize}
\item[(H1)] Cost functions of the major player: 
\be\label{controlU0}
\begin{array}{c}
\text{The control set  $(U^0,d^0)$ for the major player is a compact convex subset}\\
\text{of a finite dimensional space, not reduced to a singleton, and }\\
\text{for $i=1, \dots, I$,} \qquad L^0_i:[0,T]\times U^0\times \mathcal P_1(\R^d)\to \R \; \text{is continuous and bounded.}
\end{array}
\ee
\item[(H2)] Regularity and monotonicity of the cost functions of the small players: 
\be\label{monotonicity}
%\tag{MFG6}
\begin{cases}
\text{ $F_i, G_i:\R^d\times \mathcal P_1(\R^d)\to \R$ are Lipschitz continuous and bounded,}\\
\text{ $\sup_{m\in {\mathcal P_1(\R^d)}} \|F_i(\cdot, m)\|_{C^{2+\alpha}}+\|G_i(\cdot, m)\|_{C^{2+\alpha}} \leq C$, for $C,\alpha>0$,}\\
\text{ $F_i$ and $G_i$ \ are strongly monotone, and $ F_i$   is strictly monotone.}
\end{cases}
\ee
\item[(H3)] Regularity of the Hamiltonians of the small players: for each $i\in \{1, \dots, I\}$, the map \\ 
$H_i:\R^d\times \R^d\to \R$ satisfies:
\begin{equation}\label{A:MAINHamiltonian0}
\text{there exists $C>0$ such that, for all $x,\xi\in \R^d$,  $C^{-1}|\xi|^2-C\leq H_i(x,\xi)\leq C(|x|^2+1)$}, 
\end{equation}
\begin{equation}\label{A:MAINHamiltonian}
	\left\{
	\begin{split}
	&\text{for all $R > 0$, }\R^d \times B_R \ni (x,\xi) \mapsto  H_i(x,\xi) \quad \text{is uniformly bounded and Lipschitz;}\\
	&\xi \mapsto  H_i(x,\xi) \quad \text{is uniformly convex},\; \text{for all $R>0$,}\; \|H(\cdot,\cdot)\|_{C^{2+\alpha}(\R^d\times B_R)}\leq C_R, 
	\end{split}
	\right.
\end{equation}
and
\begin{equation}\label{A:MAINHcondition}
	\left\{
	\begin{split}
	&\text{for some $\lambda_0,C_0 > 0$ and all $t \in [0,T]$, $\xi,\xi' \in \R^d$, and $|z| = 1$,}\\
	&|D_x  H(x,\xi)| \le C_0 + \lambda_0(\xi \cdot D_\xi  H(x,\xi) -  H(x,\xi)) \text{ and}\\
	&\lambda_0\left( D_\xi  H(x,\xi) \cdot \xi -  H(x,\xi) \right) + D^2_{\xi\xi}  H(x,\xi)\xi' \cdot \xi' + 2D^2_{\xi x} H(x,\xi) z \cdot \xi' + D^2_{xx}  H(x,\xi)z \cdot z \ge -C_0.
	\end{split}
	\right.
\end{equation}
\end{itemize}

A typical Hamiltonian satisfying the properties above is $H(x,\xi)= a(x)|\xi|^2$, where $a:\R\to \R$ is smooth and bounded above and below by positive constants. \\

Note that the (strong/strict) monotonicity assumption of the $F_i$ implies the (strong/strict)  monotonicity of the map $F(\cdot, \cdot, p)$ for any $p\in \Delta(I)$. Assumption \eqref{A:MAINHcondition}, introduced in \cite{CSS22}, ensures the solution to the Hamilton-Jacobi equation in the MFG system to be wellposed and plays a key role in the construction of a solution to this MFG system.  

\subsection{The mean field game system} 

We will use throughout the paper MFG equilibria and their relationships with the MFG system. Let $(\Omega, \mathcal F, \P)$ be  a probability space endowed with a filtration $ (\mathcal F_t)_{t\in [0,T]}$ satisfying the usual conditions. Let $p=(p_t)$ be a c\`{a}dl\`{a}g process taking values in $\Delta(I)$ and adapted to $ (\mathcal F_t)$. As explained below,  $p_t$ can be interpreted as the information  available at time $t$ to the minor players.

We are interested in the MFG system 
\be\label{eq.mfgintro}
\left\{\begin{array}{l}
d  \phi_t(x) = \left\{ - \Delta  \phi_t (x)+H(x,D \phi_t(x),p_t)-F(x,m_t,p_t)\right\}dt +d M_t(x)\qquad \text{in} \; (0,T)\times \R^d\\
dm_t(x) = \left\{ \Delta m_t(x) +{\rm div}(H_\xi(x,D \phi_t(x), p_t) m_t(x))\right\}dt \qquad \text{in} \; (0,T)\times \R^d \\
 m_0(x)=\bar m_0(x), \qquad  \phi_T(x)= G(x,m_T,  p_T) \qquad \text{in} \; \R^d
\end{array}\right.
\ee

The following definition is directly borrowed from \cite{CSS22}. We will see below that, under our standing assumptions, it can be  simplified. 

\begin{Definition}[\cite{CSS22}]\label{def.MFGsystintro} We say that a triple $(\phi, m, M)$ is a solution to \eqref{eq.mfgintro} on $(\Omega, \mathcal F, \P, (\mathcal F_t)_{t\in [0,T]})$ if 
\begin{itemize}
\item[(i)] $\phi: [0,T]\times \Omega \to C(\R^d)$  is a c\`{a}dl\`{a}g process adapted to the filtration $(\mathcal F_t)$, with $\phi_T(\cdot)= G(\cdot, m_T)$. 
\item[(ii)] $M:[0,T]\times \Omega \to \mathcal M_{loc}(\R^d)\cap W^{-1,\infty}  (\R^d)$  is a c\`{a}dl\`{a}g  martingale with respect to $(\mathcal F_t)$ starting at $0$, 
\item[(iii)] $m: [0,T]\times \Omega \to  \mathcal P_2(\R^d)$ is a continuous process adapted to the filtration $(\mathcal F_t)$, with $m_0= \bar m_0$, and  such that  $m_t$ has a bounded density on $\R^d$ $\P-$a.s. and for any $t\in [0,T]$, 
\item[(iv)] there exists a deterministic constant $C > 0$ such that, with probability one and for a.e. $t\in [0,T]$ 
\be\label{boundsuD2uMmintro}
\|D\phi_t\|_\infty+ {\rm ess-sup}\ m_+(D^2\phi_t) + \|M_t\|_{\mathcal M_{loc}\cap W^{-1,\infty}} +\|m_t\|_\infty \leq C.
\ee
\item[(v)]  $(\phi,M)$ satisfies, in the distributional sense on $\R^d$, for all $t\in [0,T]$ and  $\P-$a.s., the equality:
\be\label{def.HJintro}
\phi_t(x)= G(x,m_T, p_T)+\int_t^T (\Delta \phi_s(x) -H(x,D\phi_s(x), p_s)+F(x,m_s,p_s))ds +M_t(x)-M_T(x).
\ee
\item[(vi)] $\P-$a.s. and in the sense of distributions, $m$ solves the Fokker-Planck equation 
\be\label{eq.FPintro}
dm_t(x) = \left\{ \Delta m_t(x) +{\rm div}(H_\xi(x,D \phi_t(x), p_t) m_t(x))\right\}dt\qquad \text{in} \; (0,T)\times \R^d.
\ee
\end{itemize}
\end{Definition} 

We refer to \cite{CSS22} for the notation $\mathcal M_{loc}(\R^d)$ and $W^{-1,\infty}  (\R^d)$. As we explain now, as the diffusion is nondegenerate, the solution has more regularity and the notion of solution can be simplified. 

\begin{Theorem} \label{thmMFGintro} Under our standing assumptions, there exists a unique solution $(\phi, m,M)$ to \eqref{eq.mfgintro}. 
In addition, $D^2\phi_t$ and $M_t$ are absolutely continuous with a Radon-Nykodim derivative bounded in $L^\infty$: $\P-$a.s. and for every $t\in [0,T]$, 
\be\label{addregu}
\|D^2\phi_t\|_\infty+ \|M_t\|_\infty\leq C.
\ee
Finally the solution is unique in law: if $p$ and $\tilde p$ have the same law on $\Dt0$ and $(\phi, m,M)$ and $(\tilde \phi, \tilde m,\tilde M)$ are the associated solution of the MFG system, then $(m,p)$ and $( \tilde m,\tilde p)$ have the same law on $C^0([0,T], \mathcal P_2(\R^2))\times \Dt0$.
\end{Theorem} 

A consequence of the additional regularity \eqref{addregu} is that $(M_t(x))$ is, for a.e. $x\in \R^d$ a martingale (and not only a martingale measure). In addition the  Hamilton-Jacobi equation \eqref{def.HJintro} can be understood in a pointwise sense (and not only in the sense of distributions). 

\begin{proof} We only sketch the proof, which is an easy adaptation of the proof of Theorems 4.1 and 4.2 of \cite{CSS22} and of  Theorem 3.4 of \cite{CeSo22}. The starting point is a space and time discretization of the process $p$. Fix $n\in \N$ large. We discretize the set $\Delta(I)$ into $\Delta_n(I)$  a finite and increasing in $n$ subset of $\Delta(I)$ such that $\cup_n \Delta_n(I)$ is dense in $\Delta(I)$. Let $\pi^n: \Delta(I)\to \Delta_n(I)$ be a Borel measurable map such that 
$$
\lim_{n\to \infty}\sup_{q\in \Delta(I)} |\pi^n(q)-q| =0.
$$
Let  $t^n_k= kT/2^n$ for $k\in \{0, \dots, 2^n\}$ a discretization of $[0,T]$. We set
$$
p^n_t= \pi^n(p_{t^n_k})\qquad \text{for}\; t\in [t^n_k, t^n_{k+1}), \; k=0, \dots 2^n, \qquad p^n_T= \pi^n(p_T) ,
$$
and denote by $(\mathcal F^n_t)$ the complete right-continuous filtration generated by $p^n$ (it is a subfiltration of $(\mathcal F_t)$). Note that, by construction, $(\mathcal F^n_t)$ is finite. 
Proposition 4.1 of \cite{CSS22} (inspired by \cite{CaDeLa16}) implies the existence of a solution $(\phi^n, m^n, M^n)$ to \eqref{eq.mfgintro} for the filtration $({\mathcal F}^n_t)$ and for the process $(p^n_t)$ instead of $(p_t)$. In addition, the constant $C$ in \eqref{boundsuD2uMmintro} is independent of $n$ (see \cite[Lemma 2.2]{CSS22}). 

In order to pass to the limit as $n\to\infty$, we now use the argument in the proof of Theorem 3.4 of \cite{CeSo22}\footnote{The approach of \cite{CeSo22} requires stronger monotonicity assumptions on $F$ and $G$ than \cite{CSS22} but avoids the notion of weak solution of \cite{CSS22}.}. Following Lemma 3.8 in \cite{CeSo22} and its proof, we have, for any $n, n'\geq 1$,  
\begin{align*}
&  \E\Bigl[ \| G(\cdot,m^n_T,p^n_T)-G(\cdot,m^{n'}_T,p^{n'}_T)\|_\infty^{d+2} +\int_0^T\| F(\cdot, m^n_t,p^n_t)- F(\cdot, m^{n'}_t,p^{n'}_t)\|_\infty^{d+2} dt\Bigr]\\
& \qquad \leq C(\sup_{q\in \Delta(I)} |\pi^n(q)-q|+ \sup_{q\in \Delta(I)} |\pi^{n'}(q)-q|).
\end{align*}
Then, by Lemma 3.9 of  \cite{CeSo22}, we infer that the family $(\phi^n)$ is a Cauchy sequence with respect to the family of seminorms
$$
\left(\sup_{t\in [0,T]}  \left(\E\left[\| \phi_t\|_{L^\infty (B_R)}^{d+1}\right]\right)^{1/(d+1)}\right)_{R>0}
$$
 We can then conclude as in the proof of Theorem 3.4 of \cite{CeSo22} that there exists a unique solution to \eqref{eq.mfgintro}. \\
  
 We now show the extra regularity of $\phi$ and $M$. Note that, by \eqref{def.HJintro}, we have the following representation formula for the solution: 
 $$
 \phi_t(x)= \E\left[ (\Gamma(\cdot, T-t)\ast G(\cdot,m_T))(x) +\int_t^T (\Gamma(\cdot, s-t)\ast h_s)(x)ds\ |\; \mathcal F_t\right]
 $$
 where $\Gamma$ is the heat kernel and 
 $$
 h_t(x)=H(x,D\phi_t(x), p_t)-F(x,m_t,p_t),
$$ 
 which is adapted and bounded. By our regularity assumptions of $G(\cdot,m)$ (uniform in $m$) and the fact that $h$ is bounded, we infer that, for any $\alpha\in (0,1)$, 
 $$
 \sup_{t\in [0,T]} \|\phi_t( \cdot)\|_{C^{1,\alpha}} \leq C_\alpha. 
 $$
 for some constant $C_\alpha$.  This in turn (with the regularity of $H$ and $F(\cdot, m)$ and the bound on $\|D\phi_t\|_\infty$) implies that 
 $$
 \sup_{t\in [0,T]} \|h_t(\cdot)\|_{C^\alpha} \leq C_\alpha. 
 $$
 As 
 $$
 \int_{\R^d} |D^2\Gamma(x,t)||x|^\alpha dx \leq C_\alpha t^{\alpha/2-1}, 
 $$
 we deduce that 
 \begin{align*}
 |D^2 \phi_t(x)|& \leq  \| D^2G\|_\infty  +\E\left[\int_t^T \int_{\R^d} |D^2\Gamma(x-y, s-t)||h_s(y)-h_s(x)|dyds\ |\; \mathcal F_t\right] \\
&  \leq   \|D^2G\|_\infty + \sup_{(t,\omega)\in [0,T]\times \Omega} \|h_t(\cdot)\|_{C^\alpha} \int_t^T \int_{\R^d} |D^2\Gamma(x-y, s-t)||y-x|^\alpha dyds \\
&  \leq  C_\alpha( 1+  \int_t^T (s-t)^{\alpha/2-1}ds)\leq C_\alpha.
 \end{align*}
Recalling \eqref{def.HJintro} and the fact that $M_0=0$  also gives the $L^\infty$ bound on $M$. \\

The  uniqueness in law can be proved as in \cite[Theorem 4.2]{CSS22}.
 \end{proof}
 
 We now show that, given a random distribution of players, the solution to the backward Hamilton-Jacobi (HJ) equation can be interpreted as a value function: 
 
 \begin{Proposition}\label{prop.HJHJHJ} Let $(\Omega, \mathcal F, \P, (\mathcal F_t)_{t\in [0,T]})$ be a filtered probability space and $p=(p_t)$ be a c\`{a}dl\`{a}g process taking values in $\Delta(I)$ and adapted to $ (\mathcal F_t)$. Let also $(m_t)$ be a  continuous random process with values in $\mathcal P_2(\R^d)$, adapted to the filtration $ (\mathcal F_t)$. Then the HJ equation 
 \be\label{HJHJHJ}
 \left\{\begin{array}{l}
d  \phi_t(x) = \left\{ - \Delta  \phi_t (x)+H(x,D \phi_t(x),p_t)-F(x,m_t,p_t)\right\}dt +d M_t(x)\qquad \text{in} \; (0,T)\times \R^d\\
\phi_T(x)= G(x,m_T,  p_T) \qquad \text{in} \; \R^d
\end{array}\right.
\ee
has a unique solution in the sense described above. Let also $(\Omega^1, \mathbb F^1, \P^1, (\mathcal F^1_t))$ be another filtered probability space  supporting a Brownian motion $B$ and a random variable $Z$ on $\R^d$ of law $\bar m_0$ ($B$ and $Z$ being independent) and  $\alpha^*$ and $X^*$ be given by 
$$
X^*_t = Z-\int_0^t H_\xi(X^*_s, D\phi_s(X^*_s),p_s)ds + \sqrt{2}B_t, \qquad \alpha^*_t= - H_\xi(X^*_t, D\phi_t(X^*_t), m_t,p_t),\qquad  t\in [0,T] . 
$$
Then, for any control $\alpha\in L^2((0,T)\times \Omega\times \Omega_1)$ adapted to the filtration generated by $(p,m,B)$,
\begin{align*}
\E^{\P\otimes \P^1}\left[ \phi_0(Z)\right]& = \E^{\P\otimes \P^1}\left[ \int_0^T (L(X^*_s,\alpha^*_s,p_s)+F(X^*_s,m_s,p_s))ds+G(X^*_T,m_T,p_T)\right] \\
& \leq \E^{\P\otimes \P^1}\left[ \int_0^T (L(X_s,\alpha_s,p_s)+F(X_s,m_s,p_s))ds+G(X_T,m_T,p_T)\right]\\
& \qquad - C^{-1}\E^{\P\otimes \P^1}\left[ \int_0^T |\alpha_s+ H_\xi(X_s, D\phi_s(X_s),p_s)|^2ds \right],
\end{align*}
where $X_t= Z+\int_0^t \alpha_sds+   \sqrt{2}B_t,\qquad  t\in [0,T] $. 
\end{Proposition} 

\begin{proof} The proof is standard and we only sketch it.  As the probability space is fixed in the proof, we set $\E[\cdots]= \E^{\P\otimes \P^1}[\cdots]$. We discretize time, setting   $t^n_k= kT/2^n$ for $k\in \{0, \dots, 2^n\}$ and 
$$
(m^n_t,p^n_t)= (m^n_{t^n_k},p^n_{t^n_k})\qquad \text{on}\; [t^n_k, t^n_{k+1})
$$
Then we consider the solution $(\phi^n,M^n)$ to 
$$
 \left\{\begin{array}{l}
d  \phi^n_t(x) = \left\{  -\Delta  \phi^n_t (x)+H(x,D \phi^n_t(x),p^n_t)-F(x,m^n_t,p^n_t)\right\}dt +d M^n_t(x)\qquad \text{in} \; (0,T)\times \R^d,\\
\phi^n_T(x)= G(x,m^n_T,  p_T) \qquad \text{in} \; \R^d,
\end{array}\right.
$$
adapted to the complete filtration $(\mathcal F^n_t)$ generated by $(m^n,p^n)$. Then $M^n$ is constant on each interval $(t^n_k, t^n_{k+1})$ and thus  $(\phi^n,M^n)$ is a classical solution to
$$
 \left\{\begin{array}{l}
\partial_t  \phi^n_t(x) = \left\{ - \Delta  \phi^n_t (x)+H(x,D \phi^n_t(x),p^n_t)-F(x,m^n_t,p^n_t)\right\}dt \qquad \text{in} \; (t^n_k, t^n_{k+1})\times \R^d,\\
\phi^n_{(t^n_{k+1})^-}(x)= \E\left[\phi^n_{(t^n_{k+1})}(x)\ |\ \mathcal F^n_{t^n_k}\right]  \qquad \text{in} \; \R^d, \; k<2^n,\\
\phi^n_T(x)= G(x,m_T,p_T)  \qquad \text{in} \; \R^d.
\end{array}\right.
$$
Given a control $\alpha$  adapted to the filtration generated by $(p,m,B)$ and $X_t= Z+\int_0^t \alpha_sds+   \sqrt{2}B_t$, we can compute on each time interval 
\begin{align*}
& \phi^n_{(t^n_{k+1})^-}(X_{t^n_{k+1}}) = 
\phi^n_{t^n_{k}}(X_{t^n_{k}})+\int_{t^n_k}^{t^n_{k+1}} (\partial_t  \phi^n_t(X_t) + \Delta \phi^n(X_t)+ \alpha_t\cdot D\phi^n_t(X_t))dt + \int_{t^n_k}^{t^n_{k+1}}
D\phi^n_t(X_t)\cdot dB_t\\ 
& = \phi^n_{t^n_{k}}(X_{t^n_{k}})+\int_{t^n_k}^{t^n_{k+1}} (H(X_t,D \phi^n_t(X_t),p^n_t)-F(X_t,m^n_t,p^n_t)+ \alpha_t\cdot D\phi^n_t(X_t))dt + \int_{t^n_k}^{t^n_{k+1}}
D\phi^n_t(X_t)\cdot dB_t\\
& \geq \phi^n_{t^n_{k}}(X_{t^n_{k}})+\int_{t^n_k}^{t^n_{k+1}} (-L(X_t,\alpha_t,p^n_t)-F(X_t,m^n_t,p^n_t)+C^{-1} |\alpha_t+H_\xi(X_t, D\phi_t^n(X_t),p_t^n)|^2)dt  + \int_{t^n_k}^{t^n_{k+1}}
D\phi^n_t(X_t)\cdot dB_t, 
\end{align*}
with an equality only if $\alpha_t= -H_\xi(X_t,D \phi^n_t(X_t),p^n_t)$ for a.e. $t$ and a.s.. Taking expectation and summing over $k$ gives 
$$
\E\left[ G(X_T, m_T,p_T)\right] \geq \E\left[ \phi_0(Z)\right] -\E\left[\int_0^T (L(X_t,\alpha_t, p^n_t)+F(X_t, m^n_t,p^n_t)+C^{-1} |\alpha_t+H_\xi(X_t, D\phi_t^n(X_t),p_t^n)|^2)dt\right], 
$$
with an equality only if $X=X^n$ where $X^n_t= Z-\int_0^t H_\xi(X^n_s,D \phi^n_s(X^n_s),p^n_s)ds +B_t$ and $\alpha=\alpha^n= -H_\xi(X^n,D \phi^n(X^n),p^n)$. Letting $n\to\infty$ gives the result. 
\end{proof}

%%%%%%%%%%%%%%%%%%%%%%%
\subsection{Formulation of the Stackelberg problem} \label{subsec:reform}

We now come back to our original Stackelberg problem. Let us recall that $p^0\in \Delta(I)$ is the initial belief on the random variable ${\bf i}$ shared by the small players: ${\bf i}=i$ with probability $p^0_i$. \\

\noindent {\bf Admissible control of the major player.}  We denote by $\mathcal U^0([t_0,T])$ the set  of  measurable maps $u^0:[t_0,T]\to U^0$, endowed with the $L^1$-distance: 
$$
d(u^0, v^0)= \int_{t_0}^T d^0(u^0_s,v^0_s)ds \qquad \forall u^0,v^0\in \mathcal U^0([t_0,T]).
$$
We abbreviate  $\mathcal U^0([t_0,T])$ into $\mathcal U^0$ when $t_0=0$ and endow $\mathcal U^0$ with its Borel $\sigma-$algebra $\mathcal B(\mathcal U^0)$. We denote by $u^0$ its canonical process and  by $(\mathcal F^0_t)_{t\in [0,T]}$  the filtration generated by the canonical process $t\to u^0_t$. Let $\Delta(\mathcal U^0)$  be the set of random controls for the major player, that is the set of  Borel probability measures on $\mathcal U^0$. An admissible strategy for the major player is an element ${\bf u}^0=({\bf u}_i^0)_{i=1, \dots, I}$ of $(\Delta(\mathcal U^0))^I$. The interpretation of such a strategy is that the major players can choose her control in function of the index ${\bf i}$ chosen by nature. However, in order to use her information, she also needs to hide it by randomizing it. Letting $\Omega^0 =\{1, \dots, I\}$, the probability ${\bf u}^0$ generates the probability $P^{{\bf u}^0}$ on $(\Omega^0\times \mathcal U^0, \mathcal B(\Omega^0\times \mathcal U^0))$ defined by 
$$
P^{{\bf u}^0}(\{i\}\times A^0\} = p^0_i {\bf u}_i^0(A)\qquad \forall i\in \{1,\dots, I\}, \; A \in \mathcal B(\mathcal U).
$$

\noindent {\bf The mean field game equilibria}.  Given an admissible control ${\bf u}^0=({\bf u}_i^0)_{i=1, \dots, I}\in (\Delta(\mathcal U^0))^I$ of the major player, the small players observe along the time a realization $s\to u^0_s$ of  this strategy and deduce information on the random variable ${\bf i}$ from this observation. Throughout this section we fix a filtered probability space $(\Omega^1, \mathbb F^1, \P^1, (\mathcal F^1_t))$ supporting a Brownian motion $B$ and a random variable $Z$ on $\R^d$ of law $\bar m_0$. We denote by $(\mathcal F^{{\bf u^0},1}_t)$ the completion of the filtration generated by Z and by the process $t\to (u^0_t, B_t)$ with respect to the probability $P^{{\bf u}^0}\otimes \P^1$ on $(\Omega^0\times \mathcal U\times \Omega^1, \mathcal B(\Omega^0\times \mathcal U)\otimes \mathbb F^1)$. 

\begin{Definition}\label{def.Nasheq} Given an admissible control ${\bf u}^0=({\bf u}_i^0)_{i=1, \dots, I}\in (\Delta(\mathcal U^0))^I$, an MFG equilibrium associated to ${\bf u}^0$ is a pair $(\alpha^{{\bf u}^0}, m^{{\bf u}^0})$ of processes on $(\Omega^0\times \mathcal U^0\times \Omega^1, \mathcal B( \Omega^0\times U^0)\otimes \mathcal F^1, P^{{\bf u}^0}\otimes \P^1,  (\mathcal F^{{\bf u^0},1}_t))$, where $\alpha^{{\bf u}^0}$ takes its values in $\R^d$, $m^{{\bf u}^0}$ in $\mathcal P_2(\R^d)$, and 
\begin{enumerate}
\item $\alpha^{{\bf u}^0}$ is optimal in the  control problem
$$
\inf_{\alpha} \E^{P^{{\bf u}^0}\otimes \P^1} \left[ \int_0^T (L_{\bf i}(X^\alpha_t,\alpha_t)+F_{{\bf i}}(X^\alpha_t, m^{{\bf u}^0}_t))dt+G_{\bf i}(X^\alpha_T, m^{{\bf u}^0}_T)\right]
$$
where the infimum is taken over all $\R^d-$valued, $(\mathcal F^{{\bf u^0},1}_t)-$adapted controls $\alpha$ and 
$$
X^\alpha_t= Z+\int_0^t \alpha_sds +\sqrt{2}B_t, \qquad t\in [0,T],
$$

\item for any $t\in [0,T]$, $m^{{\bf u}^0}_t$ is the conditional law of $X^{\alpha^{{\bf u}^0}}_t$ given $\sigma(\{s\to {\bf u}^0_s, \; s\leq t\})$. 
\end{enumerate}
\end{Definition} 

\begin{Definition} [The major player's problem] \label{def.costoriginal} The problem of the major player then consists in minimizing over her admissible controls ${\bf u}^0\in (\Delta(\mathcal U^0))^I$ the cost 
 \be\label{pb.main}
J^0({\bf u}^0) = \sup_{(\alpha^{{\bf u}^0}, m^{{\bf u}^0})} \E^{P^{{\bf u}^0}\otimes \P^1}\bigl[ \int_0^T \sum_{i=1}^I p_i^0 L^0_i(t, u^0_s,  m_s^{{\bf u}^0})ds\Bigr], 
\ee
where the sup is taken over all MFG equilibria $(\alpha^{{\bf u}^0}, m^{{\bf u}^0})$ associated with ${\bf u}^0$. 
\end{Definition}

We take the supremum with respect to the MFG Nash equilibria  in order to handle the worst case for the informed player; however we will see that the MFG Nash equilibrium is unique (Corollary \ref{cor.mkzjnesr}), so that this supremum is actually useless. 

%%%%%%%%%%%%%%%%%%%%
%%%%%%%%%%%%%%%%%%%%
\section{The relaxed problem}

In this section we rewrite the problem for the major player as an optimal control problem over a suitable set of martingales with values in $\Delta(I)$. 

%%%%%%%%%%%%%%%%%%%%
\subsection{Definition of the relaxed problem} \label{subsec.analysisMFG}

Fix ${\bf u}^0\in (\Delta(\mathcal U^0))^I$ and denote by $(\mathcal F^{{\bf u^0}}_t)$ the completion with respect to $P^{{\bf u}^0}$ of the filtration generated by the canonical process $t\to u^0_t$ over $(\Omega^0\times \mathcal U, \mathcal B(\Omega^0\times \mathcal U))$. Knowing ${\bf u}^0$ and observing the realization ${\bf u}_{\bf i}^0$, the small players have access to the  martingale process (with values in the simplex $\Delta(I)$)
\be\label{defpu0}
{\bf p}^{{\bf u}^0}_t = \E^{{\bf u^0}} \left[ e_{{\bf i}}\ |\ \mathcal F^{{\bf u^0}}_t\right],
\ee
where $(e_i)$ is  the canonical basis on $\R^I$. Note that $({\bf p}^{{\bf u}^0}_t)$ is a $(\mathcal F^{{\bf u^0}}_t)-$martingale; we always consider its c\`{a}dl\`{a}g version. The next lemma shows how to rewrite the cost of the small players in terms of $({\bf p}^{{\bf u}^0}_t)$. 

\begin{Lemma}\label{lem.Eu0vsEP} Fix ${\bf u}^0\in (\Delta(\mathcal U^0))^I$ and let $(\Omega^1, \mathcal F^1, \P^1, (\mathcal F^1_t))$ a filtered probability space  supporting a Brownian motion $B$ and a random variable $Z$ on $\R^d$ of law $\bar m_0$ ($B$ and $Z$ being independent).  Let $(m_t)$  be a random distribution of the players, i.e., a $\mathcal P_2(\R^d)-$valued  and $(\mathcal F^{{\bf u^0}}_t)-$adapted process. Then, for any  $\R^d-$valued, $(\mathcal F^{{\bf u^0},1}_t)-$adapted control $\alpha$, and if 
$$
X^\alpha_t= Z+\int_0^t \alpha_sds +\sqrt{2}B_t, \qquad t\in [0,T], 
$$
then
\begin{align*}
&  \E^{P^{{\bf u}^0}\otimes \P^1}\left[ \int_0^T (L_{{\bf i}}(X^\alpha_t, \alpha_t)+F_{{\bf i}}(X^\alpha_t, {\bf m}_t) )dt+ G_{{\bf i}}(X^\alpha_T, {\bf m}_T)\right] \\
& \qquad =  \E^{P^{{\bf u}^0}\otimes \P^1}\left[ \int_0^T L(X^\alpha_t, \alpha_t,{\bf p}^{{\bf u}^0}_t)+F(X^\alpha_t, {\bf m}_t, {\bf p}^{{\bf u}^0}_t)  dt + G(X^\alpha_T, {\bf m}_T,{\bf p}^{{\bf u}^0}_T)\right] ,
\end{align*}
where the maps $H$, $F$ and $G$ are defined in  \eqref{def.LFGp}.
\end{Lemma}

\begin{proof}
Indeed, as $(\mathcal F^1_t)$ is independent of $(\mathcal F^{\bf u^0}_t)$,  
$$
{\bf p}^{{\bf u}^0}_t= \E\left[e_{\bf i}| \mathcal F^{\bf u^0}_t\right] = \E\left[e_{\bf i}| \mathcal F^{\bf u^0,1}_t\right], 
$$
thus we have
\begin{align*}
& \E\left[ \int_0^T (L_{{\bf i}}(X^\alpha_t, \alpha_t)+F_{{\bf i}}(X^\alpha_t, {\bf m}_t) )dt+ G_{{\bf i}}(X^\alpha_T, {\bf m}_T)\right] \\ 
&  = \int_0^T \E\left[ \E\left[  e_{\bf i}\cdot  ( (L_i(X^\alpha_t, \alpha_t)+F_i(X^\alpha_t, {\bf m}_t) ))_{i=1, \dots, I}\ |\ \mathcal F^{\bf u^0,1}_t \right] \right] dt
+ \E\left[ \E\left[  e_{\bf i}\cdot (G_i(X^\alpha_T, {\bf m}_T))_{i=1, \dots, I}\ |\ \mathcal F^{\bf u^0,1}_T \right]\right]  \\
&    = \int_0^T \E\left[  {\bf p}^{{\bf u}^0}_t \cdot  ( (L_i(X^\alpha_t, \alpha_t)+F_i(X^\alpha_t, {\bf m}_t) ))_{i=1, \dots, I} \right] dt
+ \E\left[  {\bf p}^{{\bf u}^0}_T \cdot (G_i(X^\alpha_T, {\bf m}_T))_{i=1, \dots, I}\right]  \\
&    = \int_0^T \E\left[  L(X^\alpha_t, \alpha_t,{\bf p}^{{\bf u}^0}_t)+F(X^\alpha_t, {\bf m}_t, {\bf p}^{{\bf u}^0}_t) \right] dt
+ \E\left[  G(X^\alpha_T, {\bf m}_T,{\bf p}^{{\bf u}^0}_T)\right]. 
\end{align*}
\end{proof}

\begin{Corollary}\label{cor.mkzjnesr} Under the notation and assumption of Lemma \ref{lem.Eu0vsEP}, there is a unique MFG equilibrium  $(\alpha^{{\bf u}^0}, m^{{\bf u}^0})$ associated to  ${\bf u}^0$ which is given by 
\be\label{defalphau0}
\alpha^{{\bf u}^0}_t= -H_\xi (X^*_t, D\phi_t(X^*_t),{\bf p}^{{\bf u}^0}_t)
\ee
with $X^*_t= Z-\int_0^t H_\xi (X^*_s, D\phi_s(X^*_s),{\bf p}^{{\bf u}^0}_s)ds+\sqrt{2}B_t$ and $m^{{\bf u}^0}=m$, where  $(\phi,m,M)$ is the unique solution to the MFG system \eqref{eq.mfgintro} associated to $({\bf p}^{{\bf u}^0})$ on $(\Omega^0\times \mathcal U^0, \mathcal B(\Omega^0\times \mathcal U^0), P^{{\bf u}^0},  (\mathcal F^{{\bf u^0}}_t))$. 
\end{Corollary} 

A consequence of the corollary is that  the ``sup'' in Definition \ref{def.costoriginal} is not needed. 

\begin{proof} Let $(\phi,m,M)$, $m^{{\bf u}^0}=m$,  $X^*$ and $\alpha^{{\bf u}^0}$ be as in the lemma. In view of Lemma  \ref{lem.Eu0vsEP}, we have, for any control $\alpha$ (and $X^\alpha$ defined as in that Lemma), 
\begin{align*}
&  \E^{P^{{\bf u}^0}\otimes \P^1}\left[ \int_0^T (L_{{\bf i}}(X^\alpha_t, \alpha_t)+F_{{\bf i}}(X^\alpha_t, {\bf m}^{{\bf u}^0}_t) )dt+ G_{{\bf i}}(X^\alpha_T, {\bf m}^{{\bf u}^0}_T)\right] \\
& \qquad =  \E^{P^{{\bf u}^0}\otimes \P^1}\left[ \int_0^T L(X^\alpha_t, \alpha_t,{\bf p}^{{\bf u}^0}_t)+F(X^\alpha_t, {\bf m}^{{\bf u}^0}_t, {\bf p}^{{\bf u}^0}_t)  dt + G(X^\alpha_T, {\bf m}^{{\bf u}^0}_T,{\bf p}^{{\bf u}^0}_T)\right] .
\end{align*}
Recalling Proposition \ref{prop.HJHJHJ}, we obtain 
\begin{align*}
&\E^{P^{{\bf u}^0}\otimes \P^1}\left[ \tilde  \phi_0(Z)\right] = \E^{P^{{\bf u}^0}\otimes \P^1}\left[ \int_0^T L(X^{*}_t,  \alpha^{{\bf u}^0}_t,{\bf p}^{{\bf u}^0}_t)+F(X^{*}_t, m^{{\bf u}^0}_t, {\bf p}^{{\bf u}^0}_t)  dt + G(X^{*}_T, m^{{\bf u}^0}_T,{\bf p}^{{\bf u}^0}_T)\right]\\
& \qquad \leq \E^{P^{{\bf u}^0}\otimes \P^1}\left[ \int_0^T L(X^\alpha_t, \alpha_t,{\bf p}^{{\bf u}^0}_t)+F(X^\alpha_t, m^{{\bf u}^0}_t, {\bf p}^{{\bf u}^0}_t)  dt + G(X^\alpha_T, m^{{\bf u}^0}_T,{\bf p}^{{\bf u}^0}_T)\right] \\
& \qquad \qquad  - C^{-1} \E^{P^{{\bf u}^0}\otimes \P^1}\left[\int_0^T | \alpha_t-\alpha^{{\bf u}^0}_t|^2dt\right].
\end{align*}
Therefore $\alpha^{{\bf u}^0}$ is optimal in the  control problem
$$
\inf_{\alpha} \E^{P^{{\bf u}^0}\otimes \P^1} \left[ \int_0^T (L_{\bf i}(X^\alpha_t,\alpha_t)+F_{{\bf i}}(X^\alpha_t, m^{{\bf u}^0}_t))dt+G_{\bf i}(X^\alpha_T, m^{{\bf u}^0}_T)\right]. 
$$
%Let us now check that  for any $t\in [0,T]$, $m^{{\bf u}^0}_t$ is the conditional law of $X^{\alpha^{{\bf u}^0}}_t$ given $\sigma(\{s\to {\bf u}^0_s, \; s\leq t\})$. Indeed, b
By It\^{o}'s formula and the definition of $X^{\alpha^{{\bf u}^0}}$, the conditional law of $X^{\alpha^{{\bf u}^0}}_t$ given $\sigma(\{s\to {\bf u}^0_s, \; s\leq t\})$ solves the Fokker-Planck equation in \eqref{eq.mfgintro}. By  uniqueness of the solution to this equation, we infer that it is equal to  $m^{{\bf u}^0}_t$. \bigskip

Conversely, let $(\alpha^*, m^*)$ be an MFG equilibrium  associated to  ${\bf u}^0$. Let $(\phi^*, M^*)$ be the unique solution to the HJ equation on $(\Omega^0\times \mathcal U^0, \mathcal B(\Omega^0\times \mathcal U^0), P^{{\bf u}^0},  (\mathcal F^{{\bf u^0}}_t))$ (see \cite{CSS22})
$$
\left\{\begin{array}{l}
d  \phi^*_t(x) = \left\{  \Delta   \phi^*_t (x)-H(x,D \phi^*_t(x),{\bf p}^{{\bf u}^0}_t)+ F(x,m^*_t,{\bf p}^{{\bf u}^0}_t)\right\}dt +d  M^*_t(x)\qquad \text{in} \; (0,T)\times \R^d,\\
 \phi^*_T(x)= G(x,m^*_T,  {\bf p}^{{\bf u}^0}_T) \qquad \text{in} \; \R^d.
\end{array}\right.
$$
Proposition \ref{prop.HJHJHJ} states that $\bar \alpha_t= -H_\xi(\bar X_t,D \phi^*_t(\bar X_t),{\bf p}^{{\bf u}^0}_t)$, where $\bar X$ solves $\bar X_t= Z-\int_0^t H_\xi(\bar X_s,D \phi^*_s(\bar X_s),{\bf p}^{{\bf u}^0}_s)ds+\sqrt{2}B_t$, is the unique minimizer of the problem 
\begin{align*}
& \inf_\alpha \E^{P^{{\bf u}^0}\otimes \P^1} \left[ \int_0^T (L(X^\alpha_t,\alpha_t, {\bf p}^{{\bf u}^0}_t)+F(X^\alpha_t, m^*_t), {\bf p}^{{\bf u}^0}_t)dt+G(X^\alpha_T, m^*_T, {\bf p}^{{\bf u}^0}_T)\right]\\
& = \inf_\alpha \E^{P^{{\bf u}^0}\otimes \P^1} \left[ \int_0^T (L_{\bf i}(X^\alpha_t,\alpha_t)+F_{{\bf i}}(X^\alpha_t, m^*_t))dt+G_{\bf i}(X^\alpha_T, m^*_T)\right].
\end{align*}
As $(\alpha^*, m^*)$ is an MFG equilibrium  associated to  ${\bf u}^0$, this implies that $\alpha^*_t= \bar \alpha_t$ a.s. and for a.e. $t$,  and therefore that $X^{\alpha^*}=\bar X$. On the other hand, as $m^*_t$ is the conditional law of $X^{\alpha^*}_t=\bar X_t$ given $\sigma(\{s\to {\bf u}^0_s, \; s\leq t\})$, It\^{o}'s formula implies that $m^*$ solves 
$$
\left\{\begin{array}{l}
dm^*_t(x) = \left\{ \Delta m^*_t(x) +{\rm div}(H_\xi(x,D \phi^*_t(x), {\bf p}^{{\bf u}^0}_t) m^*_t(x))\right\}dt \qquad \text{in} \; (0,T)\times \R^d \\
 m^*_0(x)=\bar m_0(x), \qquad   \text{in} \; \R^d
\end{array}\right.
$$
Therefore the triple $( \phi^*,m^*,M^*)$ is a solution to the MFG system  associated to $({\bf p}^{{\bf u}^0})$. By uniqueness we conclude that  $( \phi^*,m^*,M^*)=(\phi,m,M)$. This completes the proof. 
\end{proof}

Recall that $\Dt0$ is the set of c\`{a}dl\`{a}g functions from $\R$  to  $\D(I)$, endowed with the Meyer-Zheng topology.  Let $t\mapsto  p(t)$ be the coordinate mapping on $\Dt0$, $\mathcal G$ be the Borel $\sigma-$algebra on $\Dt0$ and  $({\cal G}_t)$ be the filtration generated  by $t\mapsto p(t)$. Given $p_0\in \Delta(I)$, we denote by $\M(p_0)$ the set of probability measures ${\bf P}$ on $\Dt0$ such that, under $\bf P$, $(\p(t),t\in\R)$ is a martingale and $p_t=p_0$ for $t\leq 0$. 
Finally for any measure ${\bf P}$ on $\Dt0$, we denote by ${\mathcal F}^{\bf P}$ the completion of the filtration $({\cal G}_t)$ with respect to ${\bf P}$ and by 
$\E^{{\bf P}}[\dots]$ the expectation with respect to ${\bf P}$.

\begin{Definition}[The relaxed problem]\label{def.relaxed} The relaxed problem is
$$
\min_{{\bf P}\in \M(p_0)} \bar J^0({\bf P})\qquad {\rm where} \qquad  \bar J^0({\bf P}):=\E^{\bf P}\bigl[ \int_0^T    \min_{u^0\in U^0} L^0(s, u^0, m^{\bf P}_s,p_s)ds\Bigr], 
$$
where $(\phi^{\bf P},m^{\bf P}, M^{\bf P})$ is the unique solution to the MFG system  \eqref{eq.mfgintro}  on $(\Dt0, \mathcal G, {\bf P}, ({\mathcal F}^{\bf P}_t))$ (in the sense of Definition \ref{def.MFGsystintro}):
\be\label{eq.MFGsyst}
\left\{\begin{array}{l}
d  \phi^{\bf P}_t(x) = \left\{  -\Delta  \phi^{\bf P}_t (x)+H(x,D \phi^{\bf P}_t(x),p_t)- F(x,m^{\bf P}_t,p_t)\right\}dt +d M^{\bf P}_t(x)\qquad \text{in} \; (0,T)\times \R^d\\
dm^{\bf P}_t(x) = \left\{ \Delta m^{\bf P}_t(x) +{\rm div}(H_\xi(x,D \phi^{\bf P}_t(x), p_t) m^{\bf P}_t(x))\right\}dt\; \text{in} \; (0,T)\times \R^d \\
 m^{\bf P}_0(x)=\bar m_0(x), \qquad \phi^{\bf P}_T(x)= G(x,m^{\bf P}_T,  p_T) \qquad \text{in} \; \R^d
\end{array}\right.
\ee
\end{Definition}

We now explain the link between the original problem for the major player (see Definition \ref{def.costoriginal}) and the relaxed problem (Definition \ref{def.relaxed}). 

\begin{Proposition} \label{prop.equiv} Let ${\bf u}^0\in (\Delta(\mathcal U^0))^I$, ${\bf p}^{{\bf u}^0}$   be given by \eqref{defpu0} and ${\bf P}$ be its law on $\Dt0$. 
Then 
$$
\bar J^0({\bf P})\leq J^0({\bf u}^0),
$$
where $J^0$ is defined in \eqref{pb.main}. 

Conversely,  given ${\bf P}\in \M(p_0)$, there exists a sequence $\bar u^{0,n}\in (\Delta (\mathcal U^0))^I$ such that 
$$
\lim_n J^0({\bf u}^{0,n})=\overline J^0({\bf P}).
$$
\end{Proposition}

\begin{proof} Let ${\bf u}^0\in (\Delta(\mathcal U^0))^I$, $({\bf p}^{{\bf u}^0})$ be given by \eqref{defpu0} and ${\bf P}$ be its law on $\Dt0$.  Let $(\phi^{\bf P},m^{\bf P}, M^{\bf P})$ be the solution of the MFG system \eqref{eq.MFGsyst} on $(\Dt0, \mathcal G, {\bf P}, ({\mathcal F}^{\bf P}_t))$. Assume that $(\alpha^{{\bf u}^0}, m^{{\bf u}^0})$ is an MFG equilibrium associated to  ${\bf u}^0$ and let $(\phi,m,M)$ be the solution to the MFG system \eqref{eq.mfgintro} associated to $({\bf p}^{{\bf u}^0})$ on $(\Omega^0\times \mathcal U^0, \mathcal B(\Omega^0\times \mathcal U^0), P^{{\bf u}^0},  (\mathcal F^{{\bf u^0}}_t))$.  We know by Corollary \ref{cor.mkzjnesr} that  $m^{{\bf u}^0}=m$ and that
$$
\alpha^{{\bf u}^0}_t= -H_\xi (X^*_t, D\phi_t(X^*_t),{\bf p}^{{\bf u}^0}_t)
$$
where $X^*_t= Z-\int_0^t H_\xi (X^*_s, D\phi_s(X^*_s),{\bf p}^{{\bf u}^0}_s)ds+\sqrt{2}B_t$. Now, as ${\bf p}^{{\bf u}^0}$ has law ${\bf P}$, we know from Theorem \ref{thmMFGintro} that $(m^{\bf P},  p)$  has the same law as $(m,{\bf p}^{{\bf u}^0})$, which proves that 
\begin{align*}
J^0({\bf u}^0) &= \E^{P^{{\bf u}^0}\otimes \P^1}\bigl[ \int_0^T \sum_{i=1}^I p_i^0 L^0_i(t, u^0_s,  m_s^{{\bf u}^0})ds\Bigr]
 = 
\E^{P^{{\bf u}^0}\otimes \P^1}\bigl[ \int_0^T L^0(t, u^0_s,  m_s^{{\bf u}^0},{\bf p}^{{\bf u}^0}_s )ds\Bigr]\\
& \geq \E^{P^{{\bf u}^0}\otimes \P^1}\bigl[ \int_0^T \min_{u^0\in U^0} L^0(t, u^0,  m_s^{{\bf u}^0},{\bf p}^{{\bf u}^0}_s )ds\Bigr] 
= \E^{\bf P}\bigl[ \int_0^T    \min_{u^0\in U^0} L^0(s, u^0,m^{\bf P}_s, p_s)ds\Bigr] = \bar J^0({\bf P}). 
\end{align*}
Hence $\bar J^0({\bf P})\leq  J^0({\bf u}^0)$. 
 
Conversely,  let ${\bf P}\in \M(p_0)$ and $(\phi^{\bf P},m^{\bf P}, M^{\bf P})$ be the solution of the MFG system \eqref{eq.MFGsyst} on $(\Dt0, \mathcal G, {\bf P}, ({\mathcal F}^{\bf P}_t))$. For each $n \in \mathbb N^*$, we introduce the regular subdivision $\Delta^n = \left\{0=t_0 < t_1 < \ldots <t_n=T \right\}$ of step $\frac Tn$ of the interval $[0,T]$. We also introduce $\epsilon_n = \frac T{2^n}.$\\

Let us consider $I$ distinct elements of $\mathcal U^0$, denoted respectively by $a_1, \ldots,  a_I \in \mathcal U^0$. For each $n\in \N^*$, we introduce the piecewise constant control $\bf u^{0,n}$ as follows:\\

for each $k\in \{0, \ldots, n-1\}$, for each $t\in [t^n_k+\epsilon_n, t^n_{k+1})$, 
$${\bf u}^{0,n}_t = \underset{u^0 \in \mathcal U^0}{\arg \min}\ L^0(t^n_k, u^0,m^{\bf P}_{t^n_k}, p_{t^n_k})$$
and for each $i\in \{1, \ldots, I\}$, for each $t\in \left[t^n_{k} + \left( \overset{i-1}{\underset{j=0}{\sum}}(p_{t^n_k})_i\right)\epsilon_n, t^n_{k} + \left( \overset{i}{\underset{j=0}{\sum}} (p_{t^n_k})_i\right)\epsilon_n  \right),$
$${\bf u}^{0,n}_t = a_i. $$
The above definition allows to encode the values of $\left(p_{t^n_k}\right)_k$ in the process $\bf u^{0,n}$, such that for each $k \in \{0, \ldots, n-1\}$,
$$\sigma \left({\bf u}^{0,n}_t, t<t^n_{k+1} \right) = \sigma \left(p_{t^n_l}, l<k+1 \right).$$
It is then clear that for each $k \in \{0, \ldots, n-1\}$ and $t\in [t^n_{k}+\epsilon_n, t^n_{k+1})$,
$${\bf p}^{{\bf u}^{0,n}}_t = p_{t^n_{k}} .$$
Finally, by Assumption (H1), we can conclude that
\begin{align*}
J^0({\bf u}^{0,n}) &=
\E^{P^{{\bf u}^{0,n}}\otimes \P^1}\!\Bigl[ \int_0^T \!\!\!L^0(t, {\bf u}^{0,n}_s,  m_s^{{\bf u}^{0,n}},{\bf p}^{{\bf u}^{0,n}}_s )ds\Bigr] \xrightarrow{n\rightarrow +\infty} \E^{\bf P}\Bigl[ \int_0^T \!\!   \min_{u^0\in U^0}\! L^0(s, u^0,m^{\bf P}_s, p_s)ds\Bigr] = \bar J^0({\bf P}). 
\end{align*}
\end{proof}

\subsection{Existence of a relaxed solution}

\begin{Theorem}\label{thm.existence} Under our standing assumptions, there exists a minimum to $\bar J^0$. 
\end{Theorem}

\begin{proof} Let $({\bf P}^n)$ be a minimizing sequence in ${\bf M}(p_0)$ for the relaxed problem (Definition \ref{def.relaxed}). Let $(\phi^n,m^n, M^n)$ be the unique solution to the MFG system \eqref{eq.MFGsyst}  on $({\bf D}, \mathcal G, {\bf P}^n, (\mathcal F^{{\bf P}^n}_t))$.

Recall that ${\bf M}(p_0)$ is compact, when $\Dt0$ is endowed with the Meyer-Zheng topology. Hence  $({\bf P}^n)$ has a converging subsequence, denoted in the same way; let ${\bf P}\in \Dt0$ be its limit. Let also $(\phi,m, M)$ be the unique solution to the MFG system \eqref{eq.MFGsyst}  on $({\bf D}, \mathcal G, {\bf P}, (\mathcal F^{P^n}_t))$. Let finally  $\gamma^n$ be an optimal coupling between ${\bf P}^n$ and ${\bf P}$. We claim that 
\be\label{zlhjakbenrdfg}
\lim_n \int_0^T {\bf d}_1(m^n(t,p),m(t,p'))\gamma^n(dp, dp')  =0. 
\ee
Indeed, according to Lemma 4.5. of \cite{CSS22}, we have, for any $R>0$,  
$$
\lim_n\sup_{t\in [0,T]} \int_{\Dt0\times \Dt0} \sup_{|x|\leq R} |\phi^n_t(x,p)-\phi_t(x,p')|\gamma^n(dp, dp') =0
$$
and 
$$
\lim_n  \int_{\Dt0\times \Dt0} \int_0^T\int_{B_R}  |D\phi^n_t(x,p)-D\phi_t(x,p')|dxdt \gamma^n(dp, dp') =0.
$$
Let $(\Omega^1, \mathcal F^1, \P^1, (\mathcal F^1_t))$ be a filtered probability space supporting a Brownian motion $B$ and a random variable $Z$ on $\R^d$ of law $\bar m_0$ (independent of $p$). On $(\Dt0\times\Dt0\times \Omega^1, \mathcal G\otimes \mathcal G\otimes \mathcal F^1, {\bf P}^n\otimes {\bf P}\otimes \P^1, (\mathcal F^{{\bf P}^n\otimes {\bf P}\otimes \P^1}_t))$ (where  the filtration $(\mathcal F^{{\bf P}^n\otimes {\bf P}\otimes \P^1}_t)$ is the completion of $(\mathcal F^{{\bf P}^n}_t\otimes\mathcal F^{{\bf P}}_t\otimes \mathcal F^1_t))$ satisfying the usual conditions), let $X^n$ solve the SDE
$$
X^n_t (p,p',\omega)= Z(\omega)-\int_0^t H_\xi(X^n_s(p,p',\omega), D\phi^n_s(X^n_s(p,p',\omega),p), p_s)ds +\sqrt{2}B_t(\omega), 
$$
and $X$ be the solution to
$$
X_t (p,p',\omega)= Z(\omega)-\int_0^t H_\xi(X_s(p,p',\omega), D\phi_s(X_t(p,p',\omega),p), p'), p'_t)ds +\sqrt{2}B_t(\omega). 
$$
Then $m^n_t$ is the conditional law of $X^n_t$  given $(\mathcal F^{{\bf P}^n}_t)$ while $m_t$ 
is the conditional law of $X_t$ given $(\mathcal F^{{\bf P}}_t)$.
 Hence, as $\|D^2\phi\|_\infty\leq C$, 
$$
\E^{{\bf P}^n\otimes {\bf P}\otimes \P^1}\left[ |X^n_t-X_t|\right] \leq  C \E^{{\bf P}^n\otimes {\bf P}\otimes \P^1}\left[\int_0^t (|X^n_s-X_s| + |D\phi^n_s(X_s, p)-D\phi_s(X_s, p')|)ds \right], 
$$
so that, by Gronwall's inequality and for any $R>0$, 
\begin{align*}
\E^{{\bf P}^n\otimes {\bf P}\otimes \P^1}\left[ |X^n_t-X_t|\right] &\leq   C\int_0^T   \E^{{\bf P}^n\otimes {\bf P}\otimes \P^1}\left[ |D\phi^n_s(X_s, p)-D\phi_s(X_s, p')|{\bf 1}_{|X_s|\leq R} \right]ds +CR^{-1} \\
& \leq C \int_0^T\int_{\Dt0\times \Dt0} \int_{B_R} |D\phi^n_s(x, p)-D\phi_s(x, p')|m_s(dx,p')\gamma^n(dp,dp')ds+ CR^{-1} \\
& \leq C\int_0^T \int_{\Dt0\times \Dt0} \int_{B_R} |D\phi^n_s(x, p)-D\phi_s(x, p')|dx\gamma^n(dp,dp')ds+ CR^{-1}
 \end{align*}
where we used for the first inequality that $X_t$ has a uniformly bounded second order moment, the fact that $m_t$ is the  conditional law of $X_t$ given $(\mathcal F^{{\bf P}}_t)$ for the second one and that $m_t$ has a uniformly bounded density for the last one.  This shows that 
\begin{align*}
& \int_{\Dt0\times \Dt0}\int_0^T {\bf d}_1(m^n_t(p), m_t(p'))dt \ \gamma^n(dp,dp')   \leq \int_0^T  \E^{{\bf P}^n\otimes {\bf P}\otimes \P^1}\left[ \E^{{\bf P}^n\otimes {\bf P}\otimes \P^1} \left[ |X^n_t-X_t|\; |\; \mathcal F^{{\bf P}^n}_t\otimes\mathcal F^{{\bf P}}_t \right]\right]dt \\ 
&\qquad\qquad\qquad  \leq C\int_{\Dt0\times \Dt0}\int_0^T \int_{B_R} |D\phi^n_s(x, p)-D\phi_s(x, p')|dxds\gamma^n(dp,dp')+ CR^{-1}, 
 \end{align*}
 which tends to $0$ as $n\to\infty$ and then $R\to \infty$. This proves \eqref{zlhjakbenrdfg}. 
 
 We  conclude the proof by noticing that 
 \begin{align*}
|\bar J^0({\bf P}^n)-\bar J^0({\bf P})| & \leq \E^{{\bf P}^n\otimes {\bf P}}\left[ \int_0^T  |\bar L^0(s, p_s, m_s(p))-  \bar L^0(s, p_s', m_s(p'))|ds\right] \\
& \leq \int_{\Dt0\times \Dt0} \int_0^T(|p_t-p'_t|+ {\bf d}_1(m^n_t(p), m_t(p')))dt\gamma^n(dp,dp'),
 \end{align*}
 which tends to $0$ as $n\to\infty$. As $({\bf P}^n)$ is a minimizing sequence, ${\bf P}$ is optimal.

\end{proof}

\section{Application to problems with finitely many players} 

\subsection{Statement of the main result} \label{subsec.statement}

Fix $N\in \N$, $N$ being the large number of small players. We consider Stackelberg equilibria of a differential game in which the $N$ small players interact with a major player. The major player announces a (random) strategy and the small players answer through a family of (``decentralized'') controls based on their observation of the Brownian motions of all players and of the control of the major player. For simplicity we assume as before that a strategy of the informed player is a probability measure ${\bf u}^0=({\bf u}^0_i)_{i=1, \dots, I}\in (\Delta (\mathcal U^0))^I$ on the set of controls. In particular it is independent of the Brownian motions of the players, of their strategies and of the initial positions.  

%\begin{Definition}
%
%\end{Definition}

Let us fix a filtered probability space $(\Omega, \mathcal F, (\mathcal F_t), {\bf P})$ supporting, for any $N\geq1$, a family of i.i.d. $\mathcal F_0-$measurable initial conditions $(x^{N,k}_0)_{k=1, \dots, N}$ of law $\bar m_0$ and a family of independent Brownian motions $(B^k)$, independent of the  $(x^{N,k}_0)_{k=1, \dots, N}$ .

For $\delta>0$, let 
\be\label{def.JN0}
J^{N,0}({\bf u}^0,\delta) = \sup_{(\alpha^{N,j})\ \delta-{\text Nash}}   \E^{{\bf P}^{{\bf u}^0}\otimes {\bf P}}\left[ \int_0^T  L^0_{\bf i}(t, u^0_s,m^N_{{\bf X}^N_s})ds\right]
\ee
where the supremum is taken over all the open-loop $\delta-$Nash equilibria $(\alpha^{N,j})$ associated to the costs 
$$
J^{N,j}({\bf u}^0,(\alpha^{N,k}))=  \E^{{\bf P}^{{\bf u}^0}\otimes {\bf P}}\left[ \int_0^T L_{\bf i}(X^{N,j}_t,  \alpha^{N,j}_t) + F_{\bf i}(X^{N,j}_t,m^{N,j}_{{\bf X}^N_t})dt+ G_{{\bf i}}(X_T^{N,j}, m^{N,j}_{{\bf X}^N_t})\right],
$$
each control $(\alpha^{N,j})$  being adapted to the filtration generated by  the initial conditions $(x^{N,k}_0)$, the control  $u^0$ and the Brownian motions $(B^k)$, and $X^{N,j}$ satisfying 
\be\label{dyn.Njoueurs}
X^{N,j}_t = x^{N,j}_0+\int_0^t \alpha^{N,j}_sds + \sqrt{2}B^j_t\qquad \forall t\in [0,T]. 
\ee
We have set ${\bf X}^N_t= (X^{N,1}_t, \dots, X^{N,N}_t)$ and $m^{N,j}_{{\bf X}^N_t}= \frac{1}{N-1}\sum_{k\neq j} \delta_{X^{N,k}_t}$. 
The fact that $(\alpha^{N,j})$ is a $\delta-$Nash equilibrium means that, for any $j\in \{1, \dots, N\}$ and any control $\alpha$ adapted to the filtration generated by the initial conditions $(x^{N,k}_0)$,  the control  $ u^0$ and the Brownian motions $(B^k)$,
$$
J^{N,j}({\bf u}^0,\alpha, (\alpha^{N,k})_{k\neq j})\geq J^{N,j}({\bf u}^0,\alpha^{N,j},(\alpha^{N,k})_{k\neq j})-\delta. 
$$
If there is no  $\delta-$Nash equilibrium for $J^{N,j}({\bf u}^0,\cdot)$, we simply set $J^{N,0}({\bf u}^0,\delta) =+ \infty$. Note that we take the $\sup$ over $\delta-$Nash equilibria in the definition of $J^{N,0}$ in order to take into account the worst possible case for the informed player. In general there are many $\delta-$Nash equilibria, so  choosing a supremum or and infimum in the definition \eqref{def.JN0} of $J^{N,0}$ makes a difference {\it a priori}. 

Our main result states that the  problem with infinitely many players described by in Section \ref{subsec:reform} gives a good approximation of the $N-$player problem when $N$ is large: 

\begin{Theorem}\label{thm.main} Assume that $L^0$, $H_i$, $F_i$ and $G_i$ satisfy conditions \eqref{controlU0}, \eqref{monotonicity}, \eqref{A:MAINHamiltonian}, \eqref{A:MAINHcondition}. Assume in addition that $\bar m_0\in \mathcal P_1(\R^d)$ has a smooth and bounded density and a finite 4th order moment. Fix $\ep>0$ and let $\bar{\bf u}^0$ be an $\ep-$minimizer for   the limit functional $J^0$ defined in \eqref{pb.main}. Then there exists $\delta>0$ and $N_0\in \N$ such that, for any $N\geq N_0$, $\bar{\bf u}^0$ is $3\ep-$optimal for $J^{N,0}(\cdot,\delta)$. 
\end{Theorem} 

\begin{Remark}{\rm 
It  can be easily seen in the proof that the result would also hold with an infimum in the definition  \eqref{def.JN0} of $J^{N,0}$ instead of a supremum. 
}\end{Remark}

The proof of Theorem \ref{thm.main} requires    two intermediate results. The first one claims the existence of $\delta-$Nash equilibria for $N$ large enough given any control ${\bf u}^0$. The second one, given again any ${\bf u}^0$, shows that the empirical measure associated to any $\delta-$Nash equilibrium is close to the corresponding MFG equilibrium. The proof of the first statement is standard and relies on arguments going back to Huang et al. \cite{huang2006large}, Carmona and Delarue \cite{carmona2013probabilistic}. It is given after the on-going proof of Theorem \ref{thm.main}. The proof of the second statement, which is inspired by similar statement by Lacker \cite{La16} and Fischer \cite{Fi17}, is much more intricate and postponed to the next section. 

\begin{Lemma} \label{lem.un} Under the assumptions of Theorem \ref{thm.main}, for any $\delta>0$, there exists $N_\delta\in \N$ such that, for any control ${\bf u}^0\in (\Delta(\mathcal U^0))^I$ and any $N\geq N_\delta$, there exists a $\delta-$Nash equilibrium for $(J^{N,j}({\bf u}^0,\cdot))$. 
\end{Lemma}

\begin{Proposition}\label{prop.Ndeltaprim} Under the assumptions of Theorem \ref{thm.main}, fix $\ep>0$ small. Then there exists $\delta>0$ and $N_\delta'\geq N_\delta$ such that, for any control ${\bf u}^0\in (\Delta(\mathcal U^0))^I$,  for all $N\geq N'_\delta$, and any $\delta-$Nash equilibrium $(\alpha^{N,j})$ for $(J^{N,j}({\bf u}^0,\cdot))$ satisfies,
$$
\sup_{t\in [0,T]} \E^{{\bf P}^{{\bf u}^0}\otimes {\bf P}}\left[ {\bf d}_1(m^N_{{\bf X}^N_t}, {\bf m}^{{\bf u}^0}_t))\right] \leq \ep, 
$$
where $m^N_{{\bf X}^N_t}= \frac1N \sum_{k=1}^N \delta_{X^{N,k}_t}$ is the empirical distribution of the positions $X^{N,k}_t$ defined in \eqref{dyn.Njoueurs} and where
$(\phi^{{\bf u}^0},{\bf m}^{{\bf u}^0},{\bf M}^{{\bf u}^0})$ is the  solution to the MFG system  \eqref{eq.mfgintro} associated to $({\bf p}^{{\bf u}^0})$ on $(\Omega^0\times \mathcal U^0, \mathcal B(\Omega^0\times \mathcal U^0), P^{{\bf u}^0},  (\mathcal F^{{\bf u^0}}_t))$, $({\bf p}^{{\bf u}^0})$ being defined by \eqref{defpu0}. 
\end{Proposition}

\begin{proof}[Proof of Theorem \ref{thm.main}] Fix $\ep>0$ and let $\delta$, $N_\delta'$ as in Proposition \ref{prop.Ndeltaprim} for $\ep/(2C)$, where $C$ is the Lipschitz constant of $L^0$ with respect to $m$. Fix $N\geq N_\delta'$ and let  ${\bf u}^0$ be a control for the major player. Let  $(\alpha^{N,j})$ be any $\delta-$Nash equilibrium for $J^{N,j}({\bf u}^0,\cdot)$ and $(\bar \alpha^{N,j})$ be a $\delta-$Nash equilibrium for $J^{N,j}(\bar{\bf u}^0,\cdot)$ which is $\ep-$optimal for $J^{N,0}(\bar{\bf u}^0, \cdot)$ in \eqref{def.JN0}: such $\delta-$Nash equilibria exist thanks to Lemma \ref{lem.un}. We denote by $X^{N,j}$ the solution associated to $\alpha^{N,j}$ and  $\bar X^{N,j}$ the solution associated to $\bar \alpha^{N,j}$. Finally, $(\phi^{{\bf u}^0},{\bf m}^{{\bf u}^0},{\bf M}^{{\bf u}^0})$ is the  solution to the MFG system  \eqref{eq.mfgintro} associated to $({\bf p}^{{\bf u}^0})$ on $(\Omega^0\times \mathcal U^0, \mathcal B(\Omega^0\times \mathcal U^0), P^{{\bf u}^0},  (\mathcal F^{{\bf u^0}}_t))$, while $(\bar \phi^{{\bf u}^0},\bar{\bf m}^{{\bf u}^0},\bar {\bf M}^{\bar {\bf u}^0})$ is the  solution to the MFG system  \eqref{eq.mfgintro} associated to $({\bf p}^{\bar {\bf u}^0})$ on $(\Omega^0\times \mathcal U^0, \mathcal B(\Omega^0\times \mathcal U^0), P^{{\bf u}^0},  (\mathcal F^{{\bf u^0}}_t))$.

Using respectively the definition of $J^{N,0}$, Proposition \ref{prop.Ndeltaprim}, the $\ep-$optimality of $\bar{\bf u}^0$ and Proposition \ref{prop.Ndeltaprim} again, we obtain, 
\begin{align*}
J^{N,0}({\bf u}^0,\delta) & \geq  \E\bigl[ \int_0^T \sum_{i=1}^I p_i L^0_i(t, {\bf u}^0_{i,s},m^N_{{\bf X}^N_s})ds\Bigr]  \geq  \E\bigl[ \int_0^T \sum_{i=1}^I p_i L^0_i(t, {\bf u}^0_{i,s},{\bf m}^{{\bf u}^0}_s)ds\Bigr]-C(\ep/(2C)) \\
& \geq \E\bigl[ \int_0^T \sum_{i=1}^I p_i L^0_i(t, \bar{{\bf u}}^0_{i,s},{\bf m}^{\bar{\bf u}^0}_s)ds\Bigr]-3\ep/2 \\
& \geq  \E\bigl[ \int_0^T \sum_{i=1}^I p_i L^0_i(t, \bar{\bf u}^0_{i,s},m^N_{\bar{\bf X}^N_s})ds\Bigr]-2\ep \geq J^{N,0}(\bar{\bf u}^0,\delta)-3\ep. 
\end{align*}
This shows the $3\ep-$optimality of $\bar{{\bf u}}^0$. 
\end{proof}

We now prove the first intermediate result. 

\begin{proof}[Proof of Lemma \ref{lem.un}] The argument is standard and relies on the conditional propagation of chaos. Fix a control ${\bf u}^0$ and let $({\bf \phi}^{{\bf u}^0}, {\bf m}^{{\bf u}^0}, {\bf M}^{{\bf u}^0})$ be the solution to the MFG system  \eqref{eq.mfgintro} associated to $({\bf p}^{{\bf u}^0})$ on $(\Omega^0\times \mathcal U^0, \mathcal B(\Omega^0\times \mathcal U^0), P^{{\bf u}^0},  (\mathcal F^{{\bf u^0}}_t))$. For $j\in \{1,\dots, N\}$, let $(\bar X^{N,j})$ be the solution to 
$$
\bar X^{N,j}_t = x^{N,j}_0- \int_0^t D_\xi H(\bar X^{N,j}_s, D\phi^{{\bf u}^0}_s(\bar  X^{N,j}_s), {\bf p}^{{\bf u}^0}_s)ds + \sqrt{2}B^j_t, \qquad t\in [0,T]
$$
and 
$$
\bar \alpha^{N,j}_s:= - D_\xi H(\bar X^{N,j}_s, D\phi^{{\bf u}^0}_s( \bar X^{N,j}_s), {\bf p}^{{\bf u}^0}_s).
$$
As the $(x^{N,j}_0)$ and the $(B^j)$ are independent, the $(\bar X^{N,j})$ are iid given $({\mathcal F}^{{\bf u}^0})$ with conditional law  ${\bf m}^{{\bf u}^0}$. By the Glivenko-Cantelli Law of large numbers (see for instance the proof of Theorem II-2.12 of \cite{CaDebook}), 
we have 
\be\label{GCthm}
\sup_{t\in [0,T]} \E\Bigl[ {\bf d}_1( m^N_{\bar{\bf X}^N_t}, {\bf m}^{{\bf u}^0}_t)\Bigr] \leq CN^{-\gamma}, 
\ee
where $C$ depends only on the data and  $\gamma\in (0,1)$ depends on the dimension $d$ only (we use here the assumption that $\bar m_0$ has a fourth order moment, which propagates to the $X^{N,j}_t$).  

Fix $j\in \{1, \dots, N\}$ and let  $\alpha$ be a control adapted to the initial conditions $(x^{N,k}_0)$, the filtration generated by the control  ${\bf u}^0$ and the Brownian motions $(B^k)$. We denote by $X^\alpha$ the solution associated with $\alpha$. Our aim is to show that 
$$
J^{N,j}({\bf u}^0,\alpha, (\bar \alpha^{N,k})_{k\neq j})\geq \E\left[ \int_0^T (L_{\bf i}(\bar X^{N,j}_t, \bar \alpha^{N,j}_t)+F_{\bf i}( \bar X^{N,j}_t, {\bf m}^{{\bf u}^0}_t))dt+ G_{{\bf i}}(\bar X^{N,j}_T, {\bf m}^{{\bf u}^0}_T)\right]- CN^{-\gamma}. 
$$
By the coercivity of $L_i$ and the  $L^\infty$ bounds on the $F_i$, $G_i$, this inequality is obvious if $\|\alpha\|_{L^2([0,T]\times \Omega)}> C$, for $C$ large enough. We now assume that 
\be\label{L2bound}
\|\alpha\|_{L^2([0,T]\times \Omega)}\leq C. 
\ee
Then
\begin{align*}
J^{N,j}({\bf u}^0,\alpha, (\bar \alpha^{N,k})_{k\neq j}) & = \E\left[ \int_0^T (L_{\bf i}(X^\alpha_t,  \alpha_t)+F_{\bf i}( X^\alpha_t,m^{N,j}_{\bar {\bf X}^N_t}))dt+ G_{{\bf i}}(X^\alpha_T, m^{N,j}_{\bar {\bf X}^{N}_t})\right]\\
& \geq \E\left[ \int_0^T (L_{\bf i}(X_t^\alpha,  \alpha_t)+F_{\bf i}( X_t^\alpha, {\bf m}^{{\bf u}^0}_t))dt+ G_{{\bf i}}(X_T^\alpha, {\bf m}^{{\bf u}^0}_T)\right]- CN^{-\gamma}\\ 
& \geq \E\left[ \int_0^T (L_{\bf i}(\bar X^{N,j}_t, \bar \alpha^{N,j}_t)+F_{\bf i}( \bar X^{N,j}_t, {\bf m}^{{\bf u}^0}_t))dt+ G_{{\bf i}}(\bar X^{N,j}_T, {\bf m}^{{\bf u}^0}_T)\right]- CN^{-\gamma}, 
\end{align*}
where the first inequality comes from  \eqref{GCthm} and the fact that $\E[{\bf d}_1(m^{N}_{\bar {\bf X}^N_t}, m^{N,j}_{\bar {\bf X}^N_t})]\leq C/N$ (thanks to the bound \eqref{L2bound}), while the second one holds by the optimality of  $\bar \alpha^{N,j}$ in the limit problem (Proposition \ref{prop.HJHJHJ}). Using once more \eqref{GCthm}, we obtain 
\begin{align*}
J^{N,j}({\bf u}^0,\alpha, (\bar \alpha^{N,k})_{k\neq j}) & \geq \E\left[ \int_0^T (L_{\bf i}(\bar  X^{N,j}_t,  \bar \alpha^{N,j}_t)+F_{\bf i}(\bar  X^{N,j}_t, m^{N,j}_{\bar{\bf X}^N_t})dt+ G_{{\bf i}}(\bar X^{N,j}_T, m^{N,j}_{\bar{\bf X}^N_T})\right]- 2CN^{-\gamma}\\
& = J^{N,j}({\bf u}^0,(\bar \alpha^{N,k})) - 2CN^{-\gamma}.
\end{align*}
This proves that $(\bar \alpha^{N,j})$ is a $\delta-$Nash equilibrium if one chooses $N\geq N_\delta$ for $N_\delta$ large enough.  
\end{proof}

\subsection{Proof of Proposition \ref{prop.Ndeltaprim}}\label{subsec.proofmainlemma}

This part is devoted to the proof of Proposition \ref{prop.Ndeltaprim}.  It relies  on the construction of Lacker \cite{La16} (see also  Fischer \cite{Fi17} and Djete \cite{Dj21}). We argue by contradiction, assuming that there exists $\ep>0$ and, for any $n$ large, a random control ${\bf u}^{0,n}\in (\Delta(\mathcal U^0))^I$, a number of players $N_n\geq n$ and a probability space $(\Omega^n, \mathcal F^n, \P^n)$ on which are defined  i.i.d. initial conditions $x^{n,j}_0\in \R^d$, $N_n$ $d-$dimensional independent Brownian motions $(B^{n,j})$, independent of the initial conditions,   and an $1/n-$Nash equilibrium $(\alpha^{N_n,j})_{j=1, \dots, N_n}$ for the payoffs $(J^{N_n,j}({\bf u}^{0,n},\cdot))$ such that 
\be\label{hypcontra}
\sup_{t\in [0,T]} \E^n\left[ {\bf d}_1(m^{N_n}_{\bar{\bf X}^{N_n}_t}, {\bf m}^{{\bf u}^{0,n}}_t))\right] > \ep,
\ee
where $\bar{\bf X}^{N_n}=(\bar X^{N_n,1}, \dots, \bar X^{N_n,N_n})$ is the trajectory associated with $(\alpha^{N_n,j})_{j=1, \dots, N_n}$ as in \eqref{dyn.Njoueurs} and $({\boldsymbol \phi}^{{\bf u}^{0,n}},{\bf m}^{{\bf u}^{0,n}},{\bf M}^{{\bf u}^{0,n}})$ is the solution to the MFG system   \eqref{eq.mfgintro} associated to $({\bf p}^{{\bf u}^{0,n}})$ on $(\Omega^0\times \mathcal U^0, \mathcal B(\Omega^0\times \mathcal U^0), P^{{\bf u}^{0,n}},  (\mathcal F^{{\bf u^{0,n}}}_t))$, $({\bf p}^{{\bf u}^{0,n}})$ being defined  as usual by
$$
{\bf p}^{{\bf u}^{0,n}}_t = \E^{{\bf u^{0,n}}} \left[ e_{{\bf i}}\ |\ \mathcal F^{{\bf u^{0,n}}}_t\right]. 
$$

We first place ourselves on the  space $\mathcal X=(C^0([0,T], \R^d))^2\times (C^0([0,T], \mathcal P_2))^2\times {\bf D}$.  Let $(\mathcal F^{\mathcal X}_t)$ be the canonical filtration on $\mathcal X$, i.e., $\mathcal F^{\mathcal X}_t$ is the $\sigma-$algebra generated by the maps 
$$
\mathcal X \ni \omega:= (x,w,m,\hat m, p)\mapsto (x_s,w_s,m_s,\hat m_s , p_s) \; \text{for}\; s\leq t . 
$$
Let 
$$
Q^n= \frac{1}{N_n}\sum_{j=1}^{N_n}  P^{{\bf u}^{0,n}}\otimes \P^n\circ (X^{N_n,j}, B^{n,j}, m^{N_n}_{{\bf X}^{N_n}}, {\bf m}^{{\bf u}^{0}_n}, {\bf p}^n)^{-1}.
$$
Assumption \eqref{hypcontra} can be rephrased as follows:  
\be\label{hypcontraBIS}
\sup_{t\in [0,T]} \int_{\mathcal X} {\bf d_1}(m_s,\hat m_s)Q^n(dx,dw,dm, d\hat m, dp) > \ep.
\ee

\begin{Lemma}\label{lem.compact} The family $(Q^n)$ is relatively compact in $\mathcal P^1(\mathcal X)$. 
\end{Lemma}

The lemma relies on standard arguments that we briefly recall. 

\begin{proof}  As $(\alpha^{N_n,j})_{j=1, \dots, N_n}$ is a $1/n-$Nash for the payoffs $(J^{N_n,j}({\bf u}^{0,n},\cdot))$, our assumptions on the running and terminal cost, as well as the growth conditions on $L$ give the existence of a constant $C$ (independent of $j$ and $n$) such that 
$$
\E\left[\int_0^T | \alpha^{N_n,j}_s|^2ds \right]\leq C\qquad \forall j, n.
$$
This implies the tightness of  $(Q^n)$ by classical arguments (see for instance \cite[section 5.2]{La16} or \cite{Fi17} in closely related frameworks). 
\end{proof}

From now on we fix $Q$  any limit (up to subsequence) of $(Q^n)$. For simplicity we argue as if $Q$ is the limit of the whole sequence $(Q^n)$. We denote by ${\mathcal F}^{Q}$ the completion of the filtration $\mathcal F^{\mathcal X}$ for $Q$.

\begin{Lemma} \label{lem.basic} In $(\mathcal X,{\mathcal F}^{Q}, Q)$, 
\begin{enumerate}
\item\label{deux}
$w$ is a Wiener process, $p$ is a c\`{a}dl\`{a}g martingale valued in $\Delta(I)$, with $ p_t=p_0$ for $t<0$ and $Q\circ x_0^{-1}=\bar m_0$. 
\item\label{quatre} $x_0$, $w$ and $(p,m,\hat m)$ are independent.
\item \label{lem.m=condlaw}  $m = Q[x \in \cdot | (m,p)]$ a.s. That is, $m$ is a version of the conditional law of $x$ given $(m,p)$. 
\end{enumerate}
\end{Lemma}

\begin{proof} 
\begin{enumerate}
\item 
On $(\mathcal X,\mathcal F^{\mathcal X}, Q)$, $w$ is a Wiener process as this is the case for $Q^n$. Moreover, using \cite{MeZh84}, $p$ is a c\`{a}dl\`{a}g martingale valued in $\Delta(I)$. Note that, $Q-$a.s.,  we have $ p_t=p_0$ for $t<0$ and $P( p_T=i)=p_{0,i}$ for any $i=1, \dots, I$ as this is the case for $Q^n$.  This implies that the same holds for the completed filtration $(\mathcal X,{\mathcal F}^{Q}, Q)$. In addition, for any continuous and bounded map $\phi:\R^d\to \R$,  
$$
\E^{Q^n}\left[ \phi( x_0)\right] = \frac{1}{N_n}\sum_{j=1}^{N_n} \E^{{\bf P}^{{\bf u}^{0,n}} \otimes \P^n}\Bigl[\phi(X^{N,j}_0)\Bigr]= \int_{\R^d}\phi(x)\bar m_0(dx).
$$ 
Thus  $Q\circ x_0^{-1}=\bar m_0$. 
\item  Let $\phi_1:\R^d\to\R$, $\phi_2:C^0\to \R$, $\phi_3:\Dt0\times (C^0_{\mathcal P_2})^2\to \R$ be continuous and bounded. Then 
\begin{align*}
\E^{Q^n} \left[ \phi_1(x_0)\phi_2(w)\phi_3(p,m,\hat m)\right] = \frac{1}{N_n}\sum_{j=1}^{N_n} \E^{{\bf P}^{{\bf u}^{0,n}}\otimes \P^n} \left[ \phi_1(X^{N,j}_0)\phi_2(B^j)\phi_3({\bf p}^n,m^{N_n}_{\bf X^{N_n}},{\bf m}^{\bf u^{0}_n})\right] 
\end{align*}
Since the $(X^{N,j}_0, B^{n,j})$ are iid with law $\bar m_0\otimes \mathcal W^d$ ($\mathcal W^d$ being the Wiener measure on $C^0$) we have, by the law of large numbers, 
$$
\lim_n\E^{{\bf P}^{{\bf u}^{0,n}}\otimes \P^n} \left[  \left|\frac{1}{N_n}\sum_{j=1}^{N_n} \phi_1(X^{N,j}_0)\phi_2(B^j)- \int_{\R^d\times C^0} \phi_1\phi_2 d\bar m_0\otimes \mathcal W^d\right|\ \left|\phi_3({\bf p}^n,m^{N_n}_{\bf X^{N_n}},{\bf m}^{\bf u^{0}_n})\right| \right]= 0. 
$$
Thus 
$$
\lim_n \E^{Q^n} \left[ \phi_1(x_0)\phi_2(w)\phi_3(p,m,\hat m)\right] = (\int_{\R^d} \phi_1 d\bar m_0)(\int_{C^0}\phi_2d\mathcal W^d)  \E^Q\left[ \phi_3(p,m,\hat m)\right] ,
$$
which shows that the  $x_0$, $w$ and $(p,m,\hat m)$ are independent.

\item This is an adaptation of Lemma 5.5  in Lacker \cite{La16}. For any $\psi_1: C^0([0,T], \mathcal P_2(\R^d))\times \Dt0\to  \R$ and $\psi_2: C^0([0,T], \R^d)\to  \R$ bounded and continuous, 
\begin{align*}
\E^Q\left[ \psi_1(m,p)\psi_2(x)\right] & = \lim_n \E^{Q^n}\left[ \psi_1(m,p) \psi_2(x)\right]  \\ 
& = \lim_n \E^{{{\bf u}^{0,n}}\otimes \P^n}\left[ \psi_1(m^{N_n}_{{\bf X}^{N_n}},{\bf p}^n)\frac{1}{N_n} \sum_{j=1}^{N_n} \psi_2( X^{N_n,j} )\right]\\
& = \lim_n \E^{{{\bf u}^{0,n}}\otimes \P^n}\left[ \psi_1(m^{N_n}_{{\bf X}^{N_n}},{\bf p}^n) \int_{C^0([0,T], \R^d)}  \psi_2( x)m^{N_n}_{{\bf X}^{N_n}}(dx)\right]\\ 
& = \E^Q\left[ \psi_1(m,p)\int_{C^0([0,T], \R^d)}  \psi_2(x)m(dx)\right]
\end{align*}
\end{enumerate}
\end{proof}

The next step consists in identifying $x$ as the solution of a SDE with square integrable drift. 

\begin{Lemma}\label{lem.buildalpha} Under $Q$, there exists an ${\mathcal F}^{Q}-$progressively measurable process $(\alpha_t)$ with values in $\R^d$ such that
$$
\int_{\mathcal X}\int_0^T  |\alpha_s|^2ds dQ<\infty, 
$$
$$
x_t-x_0-\sqrt{2}w_t= \int_0^t \alpha_sds\qquad \forall t\in [0,T],\; Q-a.s..
$$
%and $\sigma(\alpha_s,\; s\leq t)$ is conditionally independent of $\mathcal F^{x_0, w,p, m, \hat m}_T$ given $\mathcal F^{x_0, w,p, m, \hat m}_t$ where 
%$$
%\mathcal F^{x_0, w,p, m, \hat m}_t= \sigma(x_0, w_s, p_s, m(C), \hat m(C'), \; s\leq t, \; C,C'\in  \mathcal B_t(C^0([0,T], \mathcal P_2)))
%$$
%where $\mathcal B_t(C^0([0,T], \mathcal P_2))$ is the set of Borel subsets of  $C^0([0,T], \mathcal P_2)$ {\color{red} QUI NE DEPENDENT QUE DE $s\leq t$.} 
\end{Lemma}

%{\color{red} (la derniere partie de l'enonce est tres importante chez Carmona-Delarue-Lacker et chez Lacker puisque la dynamique $x$ pourrait contenir plus d'alea que $w$; il est essentiel que ces alea soient indep) - c'est la compatibility condition de \cite{CaDeLa16} apres la  Definition 3.1 - voir aussi Lemma 3.7. de \cite{CaDeLa16} pour un critere garantissant  la compatibility condition en lien avec la condition \ref{lem.m=condlaw} du lemme \ref{lem.basic} \footnote{Note Philippe: En fait Lacker reussit toujours a contourner cette preuve dans son papier en utilisant une definition equivalente (Def 4.1, Lemme 4.2, Lemme 4.3). Le point cle semble etre le Lemme 4.3 qui correspond au Lemme 3.7 de Carmona-Delarue-Lacker, et ce qui nous interesse est la seconde etape de leur preuve (p:18). Je pense qu'en utilisant notre Lemme 3.5, on peut faire exactement la meme chose qu'eux pour prouver cette independance.}. EN A-T-ON VRAIMENT BESOIN????} 

\begin{proof}  Using the quadratic growth of the Hamiltonians (see  \eqref{A:MAINHamiltonian0}), we have, for any $k\in \N$ large and any $j\in \{1, \dots, N_n\}$,  
\begin{align*}
& \E^{{{\bf u}^{0,n}}\otimes \P^n} \Bigl[ \sum_{l=0}^{k-1} \Bigr|X^{N_n,j}_{(l+1)/k}-X^{N_n,j}_{l/k}-\sqrt{2}(B^{n,j}_{(l+1)/k}-B^{n,j}_{l/k})\Bigr|^2\Bigr] \\
& = \E^{{{\bf u}^{0,n}}\otimes \P^n} \Bigl[ \sum_{l=0}^{k-1} \Bigr|\int_{l/k}^{(l+1)/k}\alpha^{N_n,j}_sds \Bigr|^2\Bigr]  \leq k^{-1}  \E\Bigl[ \int_0^T \Bigl|\alpha^{N_n,j}_s\Bigr|^2ds\Bigr] \leq Ck^{-1}.
\end{align*}
Hence, averaging over $j$ and passing to the limit:  
$$
\E^Q\Bigl[ \sum_{l=0}^{k-1} \Bigr|x_{(l+1)/k}-x_{l/k}-\sqrt{2}(w_{(l+1)/k}-w_{l/k})\Bigr|^2 \Bigr] \leq Ck^{-1}.
$$
Let us set 
$$
\alpha^m_t = k^{-1}(x_{(l+1)/k}-x_{l/k}-\sqrt{2}(w_{(l+1)/k}-w_{l/k})) \qquad \text{for}\; t\in [l/k, (l+1)/k).
$$
We have just checked that $\alpha^k$ is bounded in $L^2([0,T]\times \Omega)$. Therefore a subsequence converges weakly in $L^2$ to a random variable $\alpha\in L^2([0,T]\times \Omega)$. Fix $t\in [0,T]$ and  $A$ a  bounded random variable. Let $l_k=[kt]$. Then we have by definition of $\alpha^m$ and by the estimates above: 
$$
\E^Q[A(x_t-x_0-\sqrt{2}w_t -\int_0^t \alpha_s ds)] =\lim_{k\to\infty} \E^Q[ A(x_t-x_{l_k/k}-\sqrt{2}(w_t-w_{l_k/k})- \int_{l_k/k}^t \alpha^k_sds)] =0.
$$
Thus (using the a.s. continuity in time of all arguments), $Q-$a.s.,  
$$
x_t=x_0+\sqrt{2}w_t +\int_0^t \alpha_s ds \qquad \forall t\in [0,T].
$$
It is clear that $\alpha$ is a $\overline{\mathcal F}^{\mathcal X}-$progressively measurable process, thanks to the representation above and because $x$ and $w$ are $\overline{\mathcal F^{\mathcal X}}-$progressively measurable.  
%
%{\color{red} FINIR ET CELA RISQUE DE NE PAS ETRE SIMPLE - demander les notes a Philippe} 
%
\end{proof}

On $(\mathcal X, Q, (\overline{\mathcal F}^Q_t))$, let $(\tilde \phi^Q,\tilde M^Q)$ be the solution of the backward HJ equation 
\be\label{tildephiQ}
\left\{\begin{array}{l}
d  \tilde \phi^Q_t(x) = \left\{  -\Delta \tilde \phi^Q_t (x)+H(x,D \tilde \phi^Q_t(x),p_t)- F(x,m_t,p_t)\right\}dt +d \tilde M^Q_t(x)\qquad \text{in} \; (0,T)\times \R^d\\
\tilde \phi^Q_T(x)= G(x,m_T,  p_T) \qquad \text{in} \; \R^d
\end{array}\right.
\ee
given in Proposition \ref{prop.HJHJHJ}. 
Let also $\alpha=(x-w)'$ as defined in Lemma \ref{lem.buildalpha}. 

\begin{Lemma} We have $Q-$a.s. and for a.e. $t\in (0,T)$, 
$$
\alpha_t= -H_\xi(x_t, D \tilde \phi^Q_t(x_t),p_t)
$$
and therefore $(\tilde \phi^Q,m,\tilde M^Q)$ is a solution to the MFG system \eqref{eq.mfgintro}.
\end{Lemma}

\begin{proof} Following Proposition \ref{prop.HJHJHJ}, equation \eqref{tildephiQ} has a unique solution. In addition, one can check (using the argument in the proof of Theorem \ref{thmMFGintro} above and the construction of the solution in \cite{CSS22}) that, $Q-$a.s., 
$$
\|D^2\tilde \phi^Q\|_\infty + \|\tilde M^Q\|_\infty\leq C.
$$
Let $(X^*_t)$ be the unique solution to 
\be\label{edseds*}
X^*_t= x_0- \int_0^t H_\xi(X^*_s, D\tilde \phi^Q_s(X^*_s),p_t)ds + \sqrt{2} w_t, \qquad t\geq 0.
\ee
We set $\alpha^*_t = -H_\xi(X^*_t, D\tilde \phi^Q_t(X^*_t),p_t)$. By  construction $\alpha^*$ is bounded and adapted to the filtration $(\overline{\mathcal F}^Q_t)$. Using Lemma \ref{lem.densite} below one can find a sequence of uniformly bounded, Carath\'eodory maps $\alpha^m:[0,T]\times \mathcal X\to \R^d$, progressively measurable with respect to the filtration $(\overline{\mathcal F}_t^Q)$ such that 
$$
\lim_m \E^Q \left[ \int_0^T |\alpha_s^m-\alpha_s^*|ds\right]=0.
$$
As $(\alpha^{N_n,j})$ is a $(1/n)-$Nash equilibrium, 
\be\label{ezlrjklfg}
 J^{N_n,j}({\bf u}^{0,n},\alpha^{N_n,j}, (\alpha^{N_n,k})_{k\neq j}) 
\leq  J^{N_n,j}({\bf u}^{0,n}, \alpha^{m,n,j}, (\alpha^{N_n,k})_{k\neq j}) + (1/n),
\ee
where 
$$
\alpha^{m,n,j}_s= \alpha^m_s(X^{N_n,j}, B^{n,j}, m^{N_n}_{{\bf X}^{N_n}}, {\bf m}^{{\bf u}^{0}_n}, {\bf p}^n).
$$
Recall that 
$$
 J^{N_n,j}({\bf u}^{0,n}, \alpha^{m,n,j}, (\alpha^{N_n,k})_{k\neq j})  = \E^{\P^n}\left[ \int_0^T L_{\bf i}(X^{m,n,j}_t,  \alpha^{m,n,j}_t)+ F_{\bf i}(X^{m,n,j}_t, m^{N,j}_{{\bf X}^N_t})dt+ G_{{\bf i}}(X_T, m^{N,j}_{{\bf X}^N_t})\right]
$$
with 
$$
X^{m,n,j}_t= x_0^{n,j}+ \int_0^t \alpha^{m,n,j}_sds + \sqrt{2} B^{n,j}_t, \qquad t\in [0,T].
$$
Let $\xi^m:\mathcal X \to C^0$  be the continuous map 
$$
\xi^m(x,w,m,\hat m, p)= (t\to x_0+ \int_0^t \alpha^m_s(x,w,m,\hat m, p) ds+\sqrt{2}w_t).
$$ 
Then 
$$
\xi^m(X^{N_n,j}, B^{n,j}, m^{N_n}_{{\bf X}^{N_n}}, {\bf m}^{{\bf u}^{0}_n}, {\bf p}^n)= X^{m,n,j}
$$
and 
$$
\frac{1}{N_n}\sum_{j=1}^{N_n}  J^{N_n,j}({\bf u}^{0,n}, \alpha^{m,n,j}, (\alpha^{N_n,k})_{k\neq j}) 
= \E^{ Q^n} \left[ J(\xi^m)\right],
$$
where $J:\mathcal X\times C^0([0,T],\R^d)\to\R\cup\{+\infty\}$ is the Caratheodory map defined by  
\begin{align*}
& J(\omega, y)=J((x,w,m,\hat m, p), y)\\
& := \left\{\begin{array}{ll} \int_0^T (L(y_s,\beta_s, m_s)+F(y_s, m_s, p_s))ds + G(y_T, m_T,p_T) \quad \text{if $y-w\in H^1$, with $\beta= (y-w)'$}\\
+\infty \quad {\rm otherwise}
\end{array}\right.
\end{align*}
We often omit the first argument in $J$. Inequality \eqref{ezlrjklfg} becomes therefore 
$$
\E^{ Q^n} \left[ J(x)\right]\leq \E^{ Q^n} \left[ J(\xi^m)\right] + (1/n). 
$$
Note now that 
$$
J(\omega,\xi^m(\omega)) = \int_0^T (L(\xi_t^m(\omega), \alpha^m_t(\omega), p_t)+ F(\xi^m_t(\omega), m_t))dt+G(\xi^m_T(\omega), m_T),
$$
which is a continuous map of $\omega$. Hence we can let $n\to \infty$ and find, by lower semicontinuity of $y\to J(y)$ for the left-hand side,  
$$
\E^Q\left[J(x)\right] \leq \E^Q\left[ J(\xi^m(x))\right].
$$
We now let $m\to\infty$ to find, by dominate convergence, 
$$
\E^Q\left[J(x)\right] \leq \E^Q\left[ \int_0^T L(X^*_t, \alpha^*_t)+ F(X^*_t, m_t))dt+G(X^*_T, m_T) \right]=:I.
$$
By Proposition \ref{prop.HJHJHJ}  we have 
$$
I=\E\left[\tilde \phi^Q_0(x_0)\right] \leq \E^Q\left[J(x)\right] -C^{-1} \E\left[\int_0^T|\alpha_t+H_\xi(x_t, D\tilde \phi^Q_t(x_t))|^2dt\right],
$$
where $\alpha$ is defined in Lemma \ref{lem.buildalpha}. This implies that, $Q-$a.s., 
$$
\alpha_t= -H_\xi(x_t, D\tilde \phi^Q_t(x_t))\qquad \text{a.e.}.
$$
By the uniqueness of the solution to \eqref{edseds*}, we find therefore 
$$
X^*_t=x_t \qquad \forall t\in [0,T].
$$

Let us now prove that  $(\tilde \phi^Q,m,\tilde M^Q)$ is a solution to the MFG system. We are just left to show that $m$ satisfies in the sense of distribution the Fokker-Planck equation 
$$
dm_t(x) = \left\{ \Delta m_t(x) +{\rm div}(H_\xi(x,D \phi_t(x), p_t) m_t(x))\right\}dt\; \text{in} \; (0,T)\times \R^d. 
$$
Let $h\in C^\infty_c([0,T)\times \R^d)$ be a smooth test function. Then by It\^{o}'s formula we have 
\begin{align*}
 0& = h(T,x_T)\\
&= h(0,x_0)+\int_0^T (\partial_t h+\Delta h - H_\xi(\cdot, D\tilde \phi^Q_s(\cdot))\cdot Dh)(s,\cdot) ds+ \int_0^T \alpha_s\cdot dw_s
\end{align*}
Conditioning with respect to $(m,p)$, we find by Lemma \ref{lem.basic} (point \ref{lem.m=condlaw}) and recalling that $x_0$, $w$ and $(p,m,\hat m)$ are independent (Lemma \ref{lem.basic}, point \ref{quatre}), that, $Q-$a.s.,
$$
0= \int_{\R^d}  h(0,x) \bar m_0(dx)+ \int_0^T \int_{\R^d} (\partial_t h+\Delta h - H_\xi(\cdot, D\tilde \phi^Q_s(\cdot))\cdot Dh)(s,x) m_s(dx) ds, 
$$
which proves that $m$ solves the Fokker-Planck equation. Therefore that $(\tilde \phi^Q,m,\tilde M^Q)$ is a solution to the MFG system.
\end{proof}

To complete the proof of the previous lemma, we need to check the following: 

\begin{Lemma}\label{lem.densite} Let $\alpha$ be a bounded, $(\overline{\mathcal F}^Q_t)-$adapted control. There exists a sequence of uniformly bounded, Carath\'eodory maps $\alpha^m:[0,T]\times \mathcal X\to \R^d$, progressively measurable with respect to the filtration $(\overline{\mathcal F}_t^Q)$ such that 
$$
\lim_m \E^Q \left[ \int_0^T |\alpha_s^m-\alpha_s|ds\right]=0.
$$
\end{Lemma}

\begin{proof} Let $\alpha$ be a bounded, $(\overline{\mathcal F}^Q_t)-$adapted control. 
There exists a sequence of simple processes $\beta^n$ such that 
$$\mathbb E \left[ \int_0^T |\beta^n_t - \alpha_t| dt \right] \underset{n\rightarrow +\infty}{\longrightarrow}0$$
and for each $n \in \mathbb N$, there exists a subdivision $0=t^n_0 < t^n_1 <\ldots < t^n_{K^n}=T$ of $[0,T]$ and random variables $A^n_0, \ldots, A^n_{K^n-1}$ such that
$$\beta^n_t = \sum_{k=0}^{K^n-1}A^n_{k}{\bf 1}_{[t^n_k, t^n_{k+1})}(t),$$
where $A^n_k$ is bounded and $\overline{\FR}^Q_{t^n_k}-$measurable for each $k \in \{0, \ldots, K^n-1\}$. In other words, for each $n\in \mathbb N$, $k \in \{0, \ldots, K^n-1\}$, there exists a $\overline{\FR}^Q_{t^n_k}-$measurable function $\theta^n_k$ such that 
$$A^n_k = \theta^n_k \left( x, w, p, m  \right).$$

Let us denote by $\tilde \XR$ the set
$$\tilde \XR = (C^0([0,T], \R^d))^2\times {\bf D} \times C^0([0,T], \mathcal P_2).$$

Since $Q$ is a Borel probability measure, by the weak Lusin theorem (see \cite[Theorem 2.24]{rudin1975analyse}), for every $\epsilon >0$, there exists a compact set $H_\epsilon \subset \tilde \XR$ such that $Q \left( \tilde \XR \backslash H_\epsilon \right)<\epsilon$ and the restriction $\theta^n_k|_{H_\epsilon}$ of $\theta^n_k$ to $H_\epsilon$ is continuous. Let us write $\theta^n_{k, \epsilon} := \theta^n_k|_{H_\epsilon}$.\\

Let us denote by $\Psi$ the set of all pseudo-paths (in the sense of Meyer-Zheng \cite{MeZh84}). Then $\Psi$ is a Polish space and for every $\epsilon>0,$ $H_\epsilon \subset \tilde \XR \subset \tilde \XR_\Psi := (C^0([0,T], \R^d))^2\times \Psi \times C^0([0,T], \mathcal P_2)$. In particular, $\tilde \XR_\Psi$ is a normal space\footnote{A topological space $X$ is a normal space if, given any disjoint closed sets $E$ and $F$, there are neighbourhoods $U$ of $E$ and $V$ of $F$ that are also disjoints. This is true for any metrizable space, for instance.}, and therefore by Tietze theorem (see \cite[Theorem 35.1]{munkres2019topology}), there exists a continuous extension of $\theta^n_{k,\epsilon}$ to $\tilde \XR_\Psi$. Let us denote this extension by $\tilde \theta^n_{k,\epsilon}$. In particular, it is continuous on $\tilde \XR$ and we consider its restriction to $\tilde \XR$, that we still denote by $\tilde \theta^n_{k,\epsilon}$.\\

It is clear that
$$
\tilde \theta^n_{k,\epsilon}(x,w,p,m) \overset{L^1}{\underset{\epsilon \rightarrow 0}{\longrightarrow}}\theta^n_k \left( x, w, p, m  \right),$$
so that the process 
$$
\alpha^n_{\ep,t} =  \sum_{k=0}^{K^n-1}\tilde \theta^n_{k,\epsilon}{\bf 1}_{[t^n_k, t^n_{k+1})}(t)
$$
is Carath\'eodory and converges to $\alpha$ in the sense of the lemma as $\ep\to 0$ and then $n\to\infty$. 
\end{proof}

We now show that $\hat m$ gives rise to an MFG system. Let $(\hat \phi,\hat M)$ be the solution to 
$$
\left\{\begin{array}{l}
d  \hat \phi_t(x) = \left\{  -\Delta \hat \phi_t (x)+H(x,D \hat \phi_t(x),p_t)- F(x,\hat m_t,p_t)\right\}dt +d \hat M_t(x)\qquad \text{in} \; (0,T)\times \R^d\\
\hat \phi_T(x)= G(x,\hat m_T,  p_T) \qquad \text{in} \; \R^d
\end{array}\right.
$$

\begin{Lemma}\label{lem.limMFG} The flow of measures $\hat m$ solves $Q-$a.s. and in the sense of distributions, 
$$
d\hat m_t(x) = \left\{ \Delta \hat m_t(x) +{\rm div}(H_\xi(x,D \hat \phi_t(x), p_t) m_t(x))\right\}dt\qquad \text{in} \; (0,T)\times \R^d. 
$$
Hence $(\hat \phi, \hat m, \hat M)$ is a solution to the MFG system on $(\mathcal X, Q, (\overline{\mathcal F}^Q_t))$. 
\end{Lemma}

\begin{proof} Let $(\phi^n, M^n)$ be the solution on $(\mathcal X,\mathcal F, Q^n, \overline{\mathcal F}^{Q^n})$  to the backward HJ equation 
$$
\left\{\begin{array}{l}
d   \phi^n_t(x) = \left\{  -\Delta  \phi^n_t (x)+H(x,D  \phi^n_t(x),p_t)- F(x,\hat m_t,p_t)\right\}dt +d  M^n_t(x)\qquad \text{in} \; (0,T)\times \R^d\\
 \phi^n_T(x)= G(x,\hat m_T,  p_T) \qquad \text{in} \; \R^d
\end{array}\right.
$$
Then, $(\phi^n, \hat m, M^n)$ is a solution of the MFG system on $(\mathcal X,\mathcal F, Q^n, \overline{\mathcal F}^{Q^n})$ because the law of $\hat m$ under $Q^n$ is the same as the law of $m^{{\bf u}^{0,n}}$ under $P^{{\bf u}^{0,n}}\otimes \P^n$ (see the proof of Lemma 4.5 of \cite{CSS22}). 

Following Lemma 4.5 of  \cite{CSS22}, we know that, if $\gamma^n\in \mathcal P^1(\mathcal X\times \mathcal X)$ is an optimal coupling\footnote{Recall that, even if $\mathcal X$ is not a Polish space, it is  contained in a Polish space on which we can work.} between $Q^n$ and $Q$, then, for any $R>0$,  
$$
\lim_n\sup_{t\in [0,T]} \int_{\mathcal X\times \mathcal X} \sup_{|x|\leq R} |{\bf \phi}^n_t(x,\omega)-\hat \phi_t(x,\omega')|\gamma^n(d\omega, d\omega') =0
$$
and 
$$
\lim_n \int_{\mathcal X\times \mathcal X} \int_0^T\int_{B_R} |D{\bf \phi}^n_t(x,\omega)-D\hat \phi_t(x,\omega')|dxdt\gamma^n(d\omega, d\omega') =0.
$$
We can then conclude, as in the proof of Theorem  4.1 of \cite{CSS22} that $(\hat \phi,\hat m,\hat M)$ is a solution to the MFG system on $(\mathcal X,\mathcal F, Q, {\mathcal F}^{Q})$. 
\end{proof}

\begin{Corollary}\label{lem.key2} We have $Q[m=\hat m]=1$.  
\end{Corollary}

\begin{proof} We have proved that $(\hat \phi, \hat m, \hat M)$ and $(\tilde \phi^Q, m, \tilde M^Q)$ are two solutions of the MFG system on $(\mathcal X, Q, (\overline{\mathcal F}^Q_t))$. By the uniqueness of the solution (Theorem \ref{thmMFGintro}), we obtain $m=\hat m$ $Q-$a.s.. 
\end{proof}

\begin{proof}[Proof of Proposition \ref{prop.Ndeltaprim}]  Corollary \ref{lem.key2}  leads to a contradiction with \eqref{hypcontraBIS} and completes the proof of Proposition \ref{prop.Ndeltaprim}. 
\end{proof}

\end{document}